 \documentclass[12pt]{amsart}
\usepackage[raiselinks=false,colorlinks=true,citecolor=blue,urlcolor=blue,linkcolor=blue,bookmarksopen=true,pdftex]{hyperref}

\newtheorem{theorem}{Theorem}[section]
\newtheorem{proposition}[theorem]{Proposition}
\newtheorem{lemma}[theorem]{Lemma}
\newtheorem{remark}[theorem]{Remark}
\newtheorem{definition}[theorem]{Definition}
\newtheorem{example}[theorem]{Example}

\newcommand{\bz}{\mathbb{Z}}

\newcommand{\br}{\mathbb{R}}
\newcommand{\bc}{\mathbb{C}}
\newcommand{\bp}{\mathbb{P}}
\newcommand{\bs}{\mathbb{S}}
\newcommand{\co}{\mathcal{O}}

\newcommand{\lr}{\longrightarrow}

\newcommand{\wt}{\widetilde}
\newcommand{\im}{\textrm{Im}}
\newcommand{\re}{\textrm{Re}}
\newcommand{\alg}{\textrm{alg}}
\begin{document}
\baselineskip=15.5pt
\title[Complex structures on product of circle bundles]{Complex structures on product of circle bundles over complex manifolds\\
Structures complexes sur les produit de fibr\'e en cercles au 
dessus des vari\'et'es complexes}
\author[P. Sankaran]{Parameswaran Sankaran}
\author[A. S. Thakur]{Ajay Singh Thakur}
\address{The Institute of Mathematical Sciences, CIT
Campus, Taramani, Chennai 600113, India}
\address{Indian Statistical Institute, 8th Mile, Mysore Road, RVCE Post Bangalore 560059, India}
\email{sankaran@imsc.res.in}
\email{thakur@isibang.ac.in}
\date{}

\maketitle 

\noindent
{\bf Abstract:}  {\it Let $\bar{L}_i\lr X_i$ be a holomorphic 
line bundle over a compact complex manifold for $i=1,2$. 
 Let $S_i$ denote the associated principal circle-bundle with respect to some hermitian inner product on $\bar{L}_i$. We construct complex structures on 
$S=S_1\times S_2$ which we refer to as 
{\em scalar, diagonal,  and linear types}. While scalar 
type structures always exist, the more general diagonal but non-scalar type structures 
are constructed assuming that $\bar{L}_i$ are equivariant 
$(\bc^*)^{n_i}$-bundles satisfying 
some additional conditions. The linear type complex 
structures are constructed assuming $X_i$ are 
(generalized) flag varieties and 
$\bar{L}_i$ negative ample line bundles over $X_i$.  
When $H^1(X_1;\br)=0$ and $c_1(\bar{L}_1)\in H^2(X_1;\br)$ is non-zero, the compact manifold $S$ does not admit any symplectic structure and hence it is non-K\"ahler with respect to {\em any} complex structure.    

We obtain a vanishing theorem for $H^q(S;\mathcal{O}_S)$ when $X_i$ are projective manifolds, $\bar{L}_i^\vee$ are very ample and the cone over $X_i$ with respect to the projective imbedding defined by $\bar{L}_i^\vee$ are Cohen-Macaulay. We obtain applications to 
the Picard group of $S$. 
 When $X_i=G_i/P_i$ where $P_i$ are maximal parabolic subgroups and $S$ is endowed with linear type complex structure with `vanishing unipotent part' we show that the field of meromorphic functions on $S$ is purely transcendental over $\bc$.}

{\bf R\'esum\'e} {\it Soient $\bar{L}_i\lr X_i$ des fibr\'es en droites holomorphes sur des vari\'{e}t\'{e}s complexes compactes, pour $i=1,2$.  Soit $S_i$ le fibr\'e en cercles associ\'e par rapport \'a un produit scalaire hermitienne sur $\bar{L}_i$. On construit des structures 
complexes sur $S=S_1\times S_2$ dites des type scalaire, 
digonal, ou lin\'{e}aire. Bien que des structures du type 
scalaire existent toujours, on construit des structures 
plus g\'{e}n\'{e}rales des type diagonal mais non-scalaire 
dans le cas o\`u les $\bar{L}_i$ soient des $(\mathbb{C}^*)^{n_i}$ fibr\'{e}s \'{e}quivariants et qui 
v\'{e}rifient certaines hypotheses supplimentaires. Les 
structures complexes du type lin\'{e}aire des vari\'{e}t\'{e} des drapeaux (g\'{e}n\'{e}ralis\'{e}es 
et que les $L_i$ soient des fibr\'{e}s en droites ample 
n\'{e}gatifs.  Lorsque $H^1(X_1;\mathbb{R})=0$ et 
$c_1(\bar{L}_1)=0$ est non-nul la vari\'{e}t\'{e} compacte 
$S$ n'admet pas de structure symplectique et donc elle 
est non-K\"ahlerienne par rapport a toute structure complexe. 

On montre que $H^q(S;\mathcal{O}_S)$ s'annulle quand 
les $X_i$ sont des vari\'{e}t\'{e}s projectives, 
les $\bar{L}_i^\vee$ son tr\`{e}s amples et le c$\hat{\textrm{o}}$ne 
sur $X_i$ par rapport au plongement projectif d\'{e}fini 
par $\bar{L}_i^\vee$ sont Cohen-Macaulay.  On applique 
c'est r\'{e}sultats au groupe de Picard de $S$. Quand 
$X_i=G_i/P_i$ o\`u $P_i$ sont les 
sousgroupes paraboliques maximaux et la vari\'{e}t\'{e}     
$S$ est munie d'une structure complexe du type lin\'{e}aire avec `la partie unipotente nulle' on montre que le corps des fonctions meromorphes sur $S$ est p$\hat{\textrm{u}}$rement transcendental sur $\mathbb{C}$.} 

\section{Introduction} 
H. Hopf \cite{hopf} gave the first examples of 
compact complex manifolds which are non-K\"ahler by showing that $\bs^1\times\bs^{2n-1}$ admits a complex structure for any positive integer $n$.
Calabi and Eckmann \cite{ce} showed 
that product of any two odd dimensional spheres 
admit complex structures.   Douady \cite{douady}, Borcea \cite{borcea} and Haefliger \cite{haefliger} studied deformations of the Hopf manifolds, and, 
Loeb and Nicolau \cite{ln}, following Haefliger's ideas, 
studied the deformations of complex structures 
on Calabi-Eckmann manifolds. More recently, 
there have been many generalizations of Calabi-Eckmann 
manifolds leading to new classes of compact 
complex non-K\"ahler manifolds  by 
L\'opez de Madrano and Verjovsky \cite{lv},  Meersseman \cite{meersseman}, 
Meersseman and Verjovsky  \cite{mv}, and Bosio \cite{b}.  See also \cite{rs} and \cite{sankaran}.

In this paper we obtain another generalization  
of the classical Calabi-Eckmann manifolds.  Our 
approach is greatly influenced by the work of Haefliger \cite{haefliger} 
and of Loeb-Nicolau \cite{ln} in that the compact complex manifolds we obtain arise 
as orbit spaces of holomorphic $\bc$-actions on the   product of two holomorphic principal $\bc^*$-bundles over compact complex manifolds.  As a differentiable manifold 
it is just the product of the associated circle bundles. 
In fact we obtain a family 
of complex analytic manifolds which may be 
thought of as a deformation of the total space of a holomorphic elliptic curve bundle over the product of  the compact complex manifolds, in much the same way 
the construction of Haefliger (resp. Loeb and Nicolau) 
yields a deformation of the classical Hopf (resp.  Calabi-Eckmann) manifolds.  

The basic construction involves the notion of 
standard action by the torus $(\bc^*)^{n_1}$ on a principal $\bc^*$-bundle $L_1$ over a complex manifold $X_1$.  See Definition \ref{standard}. Let $L=L_1\times L_2$ and $X=X_1\times X_2$.
When $L_i\lr X_i$ admit standard actions, any choice 
of a sequence of complex numbers $\lambda=(\lambda_1,\ldots, \lambda_N)$, $N=n_1+n_2,$ satisfying the {\it weak hyperbolicity condition of type} $(n_1,n_2)$  (in the sense of Loeb-Nicolau \cite[p. 788]{ln}) leads to a complex structure on the product $S(L):=S(L_1)\times S(L_2)$ where $S(L_i)$ denotes the circle 
bundle over $X_i$ associated to $L_i$.   This is the 
{\it diagonal} type complex structure on $S(L)$.  The 
complex structure on $S(L)$ is obtained by identifying 
$S(L)$ as the orbit space of a certain $\bc$-action determined by $\lambda$ on $L$ and by showing that 
the quotient map $L\lr L/\bc$ is the projection of a holomorphic principal $\bc$-bundle (Theorem \ref{cbundle}). The scalar type structure arises as a special case of the diagonal type where $(\bc^*)^{n_i}=\bc^*$ is the structure group of $L_i$, $i=1,2$.   In the case of scalar type complex structure the differentiable $\bs^1\times \bs^1$-bundle with projection $S(L)\lr X$ is a holomorphic principal bundle with fibre and structure group an elliptic curve. 
  
Under the hypotheses that $H^1(X_1;\br)=0$ and $c_1(\bar{L}_1)\in H^2(X_1;\br)$ is non-zero, we show that $S(L)$ is not symplectic and is non-K\"ahler with respect to any complex structure (Theorem \ref{symplectic}).  

The construction of linear type complex structure 
is carried out under the assumption that $X_i$ is a generalized flag variety $G_i/P_i$, $i=1,2$, where $G_i$ is a simply connected semi simple linear algebraic group over $\bc$ and $P_i$ a parabolic subgroup and the associated line bundle $\bar{L}_i$ over $X_i$ is negative ample. 
In this case $L_i$ is acted on by the reductive group 
$\wt{G}_i=G_i\times \bc^*$ in such a manner that the action of a maximal torus $\wt{T}_i\subset \wt{G}_i$ on $L_i$ is standard (Proposition \ref{homogstd}).  Fix a Borel 
subgroup $\wt{B}_i\supset \wt{T}_i$ and choose an element $\lambda \in Lie(\wt{B})$ where $\wt{B}=\wt{B}_1\times \wt{B}_2\subset \wt{G}_1\times \wt{G}_2=:\wt{G}$.  Writing the Jordan decomposition $\lambda=\lambda_s+\lambda_u$ where $\lambda_s$ belongs to the the Lie algebra of $\wt{T}:=\wt{T}_1\times \wt{T_2},$     
we assume that $\lambda_s$ satisfies the weak hyperbolicity condition. For each such $\lambda$ we 
have a complex structure on $S(L)$ of {\it linear type}. 
(See Theorem \ref{lineartype}.) 
We show that $H^q(S_\lambda(L);\mathcal{O})$ vanishes for most values of $q.$   This result is valid in   greater generality; see Theorem \ref{cohomologyofs} for precise statements. We deduce that $Pic^0(S_\lambda(L))\cong \bc$, assuming that if $X_i=\bp^1$, then $\bar{L}_i$ is the generator of $Pic(\bp^1)\cong \bz$. ( See Theorem \ref{picard0}.)
When $P_i$ are maximal parabolic subgroups and $\lambda_u=0$, we show that the meromorphic function field of $S_\lambda(L)$ is a purely transcendental extension of $\bc$. 
(Theorem \ref{transcendence}).   

Our proofs  in \S2 follow mainly the ideas of Loeb and Nicolau \cite{ln}.   
The construction of linear type complex structure is a generalization of the linear type complex structures on $\bs^{2m-1}\times \bs^{2n-1}$ given in \cite{ln} to the more general context where the base space is a product of generalized 
flag varieties. We use the K\"unneth formula due to A. Cassa \cite{cassa}, besides projective normality and 
arithmetic Cohen-Macaulayness of generalized flag varieties (\cite{rr}, \cite{ramanathan}), for obtaining our results on the cohomology groups $H^q(S(L);\mathcal{O})$.  Construction of linear type complex structure, applications to Picard groups and the field of meromorphic functions on $S(L)$ when $X_i=G_i/P_i$ involve some elementary concepts from representation theory of algebraic groups. 

\begin{center}
{\bf Notations}
\end{center}
The following notations will be used throughout.

\begin{tabular}{ll}
$X_1,X_2$ & compact complex manifolds\\
$L_1,L_2$ & principal $\bc^*$-bundles\\
$L$ & $L_1\times L_2$\\
$\bar{L}_i$ & line bundle associated to $L_i$\\
$\bar{L}^\vee$ & dual of $\bar{L}$\\
$\hat{L}_i$  & cone over $X_i$ with respect to 
$\bar{L}_i^\vee$ when $\bar{L}_i$ is negative ample\\
$T_i, T, \wt{T}_i,\wt{T}$ & complex algebraic tori $T=T_1\times T_2$, 
$ \wt{T}_i=T_i\times \bc^*$, $\wt{T}=\wt{T}_1\times \wt{T}_2 $\\
$G, G_i$ & semi simple complex algebraic groups\\
$\wt{G}_i, \wt{G}$ &$\wt{G}_i=G_i\times \bc^*$, $\wt{G}=G\times \bc^*$\\
$P_i$ & a parabolic subgroup of $G_i$\\ 
$\omega_i, \omega$ & dominant integral weights\\
$V(\omega_i)$ & finite dimensional irreducible $G_i$-module of highest weight $\omega_i$\\
$V(\omega_i)^\vee$ & vector space dual to $V(\omega_i)$\\
$V(\omega_1, \omega_2)$ & $V(\omega_1)\times V(\omega_2)\setminus (V(\omega_1)\times 0\cup 0\times V(\omega_2))$\\
$\Lambda(\omega_i), \Lambda(\omega_1,\omega_2)$  & weights of $V(\omega_i)$, resp. $V(\omega_1)\times V(\omega_2)$\\
$R(G), R$ & roots of $G$ with respect to a maximal torus $T$\\
$\Phi$ & set of simple roots\\
$R^+, R_{P_i}$ & positive roots, resp. set of positive roots of $G_i$ which are not \\
 & the roots of Levi part of $P_i$\\
 $V\hat{\otimes}V'$ & completed tensor product of Fr\'echet-nuclear spaces $V,V'$\\
 $X_\beta, Y_\beta, H_\beta$ & Chevalley basis elements of a reductive Lie algebra\\
$\mathcal{O}$ & structure sheaf of an analytic space\\
$\mathcal{R}^+, \mathcal{R}^-$ & $\{x+\sqrt{-1}y\in \bc\mid x>0\}$,  resp. $\{x+\sqrt{-1}y\in \bc\mid x<0\}$\\
$I$ & unit interval $[0,1]\subset \br$\\
$\mathcal{T}_pM$ & tangent space to a manifold $M$ at a point $p$.\\

\end{tabular}

\section{Basic construction} 
Let $X_1,X_2$ be any two compact complex manifolds and let $p_1:L_1\lr X_1$ and $p_2: L_2\lr X_2$ be holomorphic principal $\bc^*$-bundles over $X_1$ and $X_2$ respectively.  Denote by 
$p:L_1\times L_2=:L\lr X:=X_1\times X_2$ the product 
$\bc^*\times \bc^*$-bundle.   
We shall denote by $\bar{L}_i$ 
the line bundle associated to $L_i$ and identify $X_i$ with the zero cross-section in $L_i$  so that $L_i=\bar{L}_i\setminus X_i$. 
We put a hermitian metric on $L_i$ invariant under the action of $\bs^1\subset \bc^*$ and denote by $S(L_i)\subset L_i$ the unit sphere bundle 
with fibre and structure group $\bs^1$. We shall denote by $S(L)$ the $\bs^1\times \bs^1$-bundle 
$S(L_1)\times S(L_2)$.  Our aim is to study complex structures on $S(L)$ arising from holomorphic 
principal $\bc$-bundle structures on $L$ with base space $L/\bc$. Such a bundle arises from the holomorphic foliation associated to certain holomorphic vector field whose integral curves are 
biholomorphic to $\bc$.     
For the vector fields we consider,  
the space of leaves $L/\bc$ can be identified with
$S(L)$ as a differentiable manifold and 
the complex structure on $S(L)$ is induced 
from that on $L/\bc$ via this identification. 

In this section we consider holomorphic $\bc$-actions on $L$ which lead to complex structure on $S(L)$ of {\it scalar} and {\it diagonal} types. 
Whereas scalar type complex structures always exist, in order to obtain the more general diagonal type complex structure which are {\it not} of scalar type we 
need additional hypotheses.   

Given any complex number $\tau$ such that $\im(\tau)>0$, one obtains a proper holomorphic imbedding  $\bc\lr \bc^*\times \bc^*$ defined as $z\mapsto (\exp(2\pi iz),\exp(2\pi i\tau z))$. We shall denote the image by $\bc_\tau$.  The action of the structure group $\bc^*\times \bc^*$ on $L$ can be restricted to $\bc$ via the above imbedding to obtain a holomorphic principal $\bc$-bundle with total space $L$ and base space the quotient space $S_\tau(L):=L/\bc_\tau$. 
Clearly the projection $L\lr X$ factors through $S_\tau(L)$ 
to yield a principal bundle $S_\tau(L)\lr X$ with fibre and 
structure group $\mathbb{E}:=(\bc^*\times \bc^*)/\bc_\tau$.  Since $\mathbb{E}$ is a Riemann surface with 
fundamental group isomorphic to $\bz^2$, it is an elliptic curve.
It can be seen that $\mathbb{E}\cong \bc/\Gamma$ where $\Gamma$ is the lattice $\bz+\tau\bz\subset \bc$.  
It is easily seen that $S_\tau(L)$ is diffeomorphic to 
$S(L)=S(L_1)\times S(L_2)$. The resulting 
complex structure on $S(L)$ is referred to as {\em scalar} 
type.

Now suppose that  $\bar{L}_i, X_i$ are acted on holomorphically by the torus group $T_i\cong (\bc^*)^{n_i}$ such that 
$\bar{L}_i$ is  a $T_i$-equivariant bundle over $X_i$.  
We identify $T_i$ with $(\bc^*)^{n_i}$ by choosing an  isomorphism $T_i\cong(\bc^*)^{n_i}$.  Set $N=n_1+n_2$ and $T:=T_1\times T_2=(\bc^*)^N$.  We shall denote by $\epsilon_j:\bc^* \subset (\bc^*)^N$ the inclusion 
of the $j$th factor and write $t\epsilon_j$ to denote $\epsilon_j(t)$ for $1\leq j\leq N$.  Thus any $t=(t_1,\cdots, t_N)\in T$ equals $\prod_{1\leq j\leq N}t_j\epsilon_j$, and, under the exponential map $\bc^N\lr (\bc^*)^N$, $\sum_{1\leq j\leq N} z_je_j$ maps to $\prod \exp(z_j)\epsilon_j$.  (Here $e_j$ denotes the 
standard basis vector of $\bc^N$.)

We put a hermitian metric on $\bar{L}_i$ which 
is invariant under action of the maximal compact subgroup $K_i=(\bs^1)^{n_i}\subset T_i$.   
The following definition will be very crucial for our construction of complex structures on $S(L)$. 

\begin{definition} \label{standard}  Let $d$ be a positive integer.
We say that the $T_1=(\bc^*)^{n_1}$-action on $L_1$ is $d$-{\em standard}  (or more briefly {\em standard}) if  the following conditions hold:\\
(i) the restricted action of the diagonal subgroup $\Delta\subset T_1$  on $L_1$ is via the 
$d$-fold covering projection $\Delta\lr \bc^*$ onto the structure group $\bc^*$ of $L_1\lr X_1$. 
(Thus if $d=1$, the action of $\Delta$ coincides with that of the structure group of $L_1$.)\\
(ii) For any $0\neq v\in L_1$ and $1\leq j\leq n_1$ let $\nu_{v,j}:\br_+\lr \br_+$ be defined as $t\mapsto ||t\epsilon_j.v||$. Then $\nu'_{v,j}(t)>0$ for all 
$t$ unless $ \br_+ \epsilon_j$ is contained in the isotropy 
at $v$. 
\end{definition}

Examples of standard actions are given in \ref{std} below.
Note that condition (i) in the above definition implies that 
the $\Delta$-orbit of any $p\in L_1$ is just the fibre of 
the bundle $L_1\lr X_1$ containing $p$. The exact 
value of $d$ will not be of much significance for us.   
However, it will be too restrictive to assume $d=1$. (See Example \ref{std}(iii) and also \S3.)

Suppose that there exists a one parameter subgroup $S\cong \bc^*$ of $T_1$ such that the 
restricted $S$ action on $L_1$ is the same as that induced by a covering $S\lr \bc^*$ of the 
structure group of $L_1$.   
Then one may parametrise $T_1$ 
so that the diagonal subgroup of $(\bc^*)^{n_1}$ maps isomorphically onto $S$.  But 
it may so happen that there exists {\it no} one-parameter subgroup $S$ satisfying condition (i) with $\Delta$ replaced by $S$. Indeed, this happens for the action of the diagonal subgroup $T_1$ of $SL(2,\bc)$ on the tautological line bundle over $\bp^{1}$.  See also 
\S3. 

Condition (ii) above controls the dynamics 
of the $T_1$-action and will have important 
implications as we shall see. 
Roughly speaking condition (ii) says that, for each $v\in L_1$, the smooth curve $\sigma_{v,j}:\br_+\lr L_1$  defined as  $t\mapsto t\epsilon_j.v$ 
always ``grows outwards" unless it is a degenerate curve.   For a counterexample, consider again 
the action of the diagonal subgroup $T_1=\{diag(t,t^{-1})\mid t\in \bc^*\}$ of $SL(2,\bc)$ on $\bc^2\setminus\{0\}\lr \bp^1$.  Then $t\epsilon_1. e_1= te_1, t\epsilon_1. e_2=t^{-1}e_2$.  Therefore $\sigma_{e_1,1}$ `grows outwards' whereas $\sigma_{e_2,1}$ `grows inwards'. On the other hand if $v=v_1e_1+v_2e_2,$ where $v_1, v_2$ are both non-zero, then the 
function $\nu_v(t)=||tv_1e_1+t^{-1}v_2e_2||=(t^2|v_1|^2+t^{-2}|v_2|^2)^{1/2}$ attains its minimum at some $t_0>0$.    

It is obvious that if $T_i=\bc^*$, the structure group 
of $L_i$, then the action of $T_i$ on $L_i=\bar{L}_i\setminus X_i$ is standard.

Let $\lambda\in Lie(T)=\bc^N$.  
There exists a unique Lie group homomorphism $\alpha_\lambda: \bc\lr T$ defined as $z\mapsto\exp(z\lambda)$.  When $\lambda $ is clear from the context, we write $\alpha$ to mean $\alpha_\lambda$. We denote by $\alpha_{\lambda,i}$ (or more briefly $\alpha_i$) the composition $\bc\stackrel{\alpha_\lambda}{\lr} T\stackrel{pr_i}{\lr}T_i, i=1,2$. 
 
We recall the definition of weak hyperbolicity \cite{ln}.
Let $\lambda=(\lambda_1,\dots,\lambda_N)$, $N=n_1+n_2$. One says that $\lambda$ satisfies the {\em weak hyperbolicity condition of type} $(n_1,n_2)$ 
 if   
 \[0\leq \arg(\lambda_i)<\arg(\lambda_j)<\pi,~1\leq i\leq n_1<j\leq N\eqno(1)\]

If $\lambda_j=1~\forall j\leq n_1, \lambda_j=\tau ~\forall j>n_1$, with $\im(\tau)>0$, we say that $\lambda$ is of {\it scalar type}. 

We denote by $C_i$  the cone $\{\sum r_j\lambda_j\in \bc
\mid r_j\geq 0, n_{i-1}+1\leq j\leq n_{i-1}+n_i\}$ where $n_0=0$. 
We shall denote $C_i\setminus\{0\}$ by $C_i^\circ$ and refer to it as  the {\it deleted cone}.
Weak hyperbolicity is equivalent to the requirement that the 
cones $C_1, C_2$ meet only at the origin and are contained in the half-space 
$\{z\in \bc\mid \im(z)>0\}\cup \br_{\geq 0}$.

\begin{definition}\label{diagonaltype}
Suppose that the $T_i=(\bc^*)^{n_i}$-action on $L_i$ is $d_i$-standard for some $d_i\geq 1$, 
$i=1,2$, and let $\lambda\in \bc^N=Lie(T)$.  The analytic homomorphism $\alpha_\lambda:\bc\lr T=(\bc^*)^N$ defined as $\alpha_\lambda(z)=\exp(z\lambda)$ is   
said to be {\em admissible} if $\lambda$ satisfies the 
weak hyperbolicity condition (1) of type $(n_1,n_2)$ above.  We denote the image of $\alpha_\lambda$ by $\bc_\lambda$. If $\alpha_\lambda$ is 
admissible, we say that the $\bc_\lambda$-action on $L$ 
is  of {\em diagonal type.} If $\lambda$ is of scalar type, we say that 
$\bc_\lambda$-action is of {\em scalar type}. 
\end{definition}

The weak hyperbolicity condition implies 
that $(\lambda_1,\dots, \lambda_N)\in \bc^N$ belongs to the Poincar\'e domain \cite{arnold}, (that is, $0$ is not in the convex hull of $\lambda_1,\ldots, \lambda_N\in \bc$,) and that $\alpha_\lambda$ is a proper holomorphic imbedding.  Thus $\bc_\lambda\cong \bc$. When there 
is no risk of confusion, we merely write $\bc$ to 
mean $\bc_\lambda$.

Note that if $\lambda$ is of scalar type, then the action of $\bc_\lambda$ leads to a scalar type 
complex structure on the orbit space $L/\bc_\lambda=S(L)$.   Moreover, if $d_1=d_2=1$, then $L/\bc_\lambda=
S_{\lambda_{n_1+1}}(L)$.

\begin{lemma} \label{norm}
Suppose that $L_1\lr X_1$ is a $T_1$-equivariant   
principal $\bc^*$-bundle such that the $T_1$-action 
is $d$-standard.  Then:\\
(i)  One has $||z\epsilon_j.v||\leq ||v||$ for $0<|z|<1$ where equality holds if  and only if $\br \epsilon_j$ is contained in the isotropy at $v$.  \\
(ii)  For any $t=(t_1,\dots,t_{n_1})\in T_1$, one has 
\[|t_{k_0}|^{d}.||v||\leq ||t.v||\leq |t_{j_0}|^{d}.||v||, \forall v\in L_1, 
\eqno(2)\]
where $j_0\leq n_1$ (resp. $k_0\leq n_1$) is  such that $|t_{j_0}|\geq |t_j|$ (resp. $|t_{k_0}|\leq |t_j|$) for all  $1\leq j\leq n_1$. Also 
$||t.v||=|t_{j_0}|^{d}.||v||$  if and only if $|t_j|=|t_{j_0}|$ for all $j$ such that $(t_j/t_{j_0})\epsilon_j.v\neq v$ and  $||t.v||=|t_{k_0}|^{d}.||v||$ if and only if $ |t_j|=|t_{k_0}|$ for all $j$ such that $(t_j/t_{k_0})\epsilon_j.v\neq v$.  \\
\end{lemma}
\begin{proof}
(i)  Suppose that $\br_+\epsilon_j$ is not contained in the isotropy at $v$.  Since $K_1$ preserves the norm, we may assume that $z\in \br_+$.  In view of \ref{standard}(ii),  $\nu_{v,j}$ is strictly increasing. Hence $||z\epsilon_j.v||<||v||$ for $0<z<1$. 

(ii) Write $s=(s_1,\dots,s_{n_1})$ where $s_j=t_j/t_{j_0}~\forall j$.   
Denoting the diagonal imbedding $\bc^*\lr T_1$ by $\delta$, we have $t=\delta(t_{j_0})s$. Now  
$\delta(t_{j_0}).v=t_{j_0}^{d}v$  in view of \ref{standard} (i). 

By repeated application of (i) above, we see 
that $||t.v||=||s(\delta(t_{j_0}) v)||=||s. t_{j_0}^{d}v||
\leq |t_{j_0}|^{d}.||v||$ where the inequality is strict unless $|t_j|=|t_{j_0}|$ for all $j$ such that $s_j\epsilon_j.v\neq v$.  A similar proof establishes 
the inequality $||t.v||\geq |t_{k_0}|^{d}.||v||$ as well as the condition for equality to hold.
\end{proof}

As an immediate corollary, we obtain

\begin{proposition}\label{free}
 Any admissible $\bc_\lambda$-action of diagonal type on $L_1\times L_2$ is free. 
\end{proposition}
\begin{proof}.
Suppose that $z\in \bc, z\neq 0,  (p_1,p_2)\in L.$ 
Let $\alpha_\lambda(z).(p_1,p_2)=(q_1,q_2)$. 
It is readily seen that one of the deleted cones $zC_1^\circ, zC_2^\circ$ lies entirely in the left-half space 
$\mathcal{R}_-:=\{z\in \bc\mid \re(z)<0\}$ or the 
right-half space $\mathcal{R}_+:=\{z\in \bc\mid \re(z)>0\}$. 
Consider the case $zC_1^\circ\subset \mathcal{R}_-$.  Then $|\exp(z\lambda_j)|  <1$ for all $j\leq n_1$.  We claim that there is some $j$ such that $\exp(z\lambda_j)\epsilon_j.p_1\neq p_1$, for, otherwise, the action of $T_1$-action, restricted to the 
orbit through $p_1$ factors through 
the compact group $T_1/\langle \exp(z\lambda_j)\epsilon_j, 1\leq j\leq n_1\rangle\cong (\bs^1)^{2n_1}$.  This implies that the $T_1$-orbit of $p_1$ is compact, contradicting \ref{standard}  (i).  
Now it follows from Lemma \ref{norm} that  $||q_1||=||(\prod_{1\leq j\leq n_1}\exp(z\lambda_j)\epsilon_j).p_1||<||p_1||$.  Thus $q_1\neq p_1$ in this case.  Similarly, we see that $(p_1,p_2)\neq (q_1,q_2)$ in the other cases also, showing that the $\bc$-action on $L$ is free. \end{proof}

\begin{example}\label{std}{\em (i) Let $T_i=\bc^*$ be the structure group of $L_i\lr X_i$ so that the $T_i$-action on $L_i$ is standard, $i=1,2$. If $\tau\in \bc^*$ is such that $0<\arg(\tau)<\pi$, then the imbedding  $\alpha(z)=(\exp(z),\exp(\tau z))\in\bc^*\times \bc^* $ is admissible.\\
(ii) Suppose that $T_1$ action on $L_1$ is $d$-standard  and that $X_1'\subset X_1$ is a $T_1$-stable 
complex analytic submanifold.  Then the $T_1$-action 
on $L_1|X_1'$ is again $d$-standard. 
More generally, 
suppose $X_1'$ is any compact complex manifold with a holomorphic $T_1$-action and that $\bar{L}'_1\lr X_1'$ is the pull-back of $\bar{L}_1\lr X_1$ via a $T_1$-equivariant  holomorphic map $f:X_1'\lr X_1$.  The hermitian metric on $\bar{L}_1$ induces a hermitian metric on $\bar{L}_1'$.  Then the 
$T_1$-action on $L_1'$ is standard.\\
(iii) Let $A=(a_{ij})$ be an $n\times n_1$ matrix---the matrix of exponents---where $a_{ij}\in\bz$.  Then $T_1:=(\bc^*)^{n_1}$ acts linearly on $\bc^n$ where $t\epsilon_j.(z_1,\ldots,z_n)=(t^{a_{1j}}z_1,\ldots,
t^{a_{n,j}}z_n), t\in \bc^*, 1\leq j\leq n_1$.  This action makes the tautological line bundle $\bar{L}_1\lr \bp^{n-1}$ a $T_1$-equivariant bundle.  The action is almost  effective if $A$ has rank equals $n_1$. Condition (i) of 
Definition \ref{standard} is satisfied if $A$ has positive constant row sums, that is, $d:=\sum_{j}a_{ij}$ is  independent of $i$ and is positive. 
Condition (ii) is satisfied if $a_{i,j}\geq 0,$ for all $1\leq i\leq n,1\leq j\leq n_1.$  Thus we obtain 
a $d$-standard $T_1$-action on $L_1$ when the matrix $A$ satisfied both these conditions where $d:=\sum_ja_{1,j}$.  \\ 
(iv) Consider the 
linear representation of $T_1\cong(\bc^*)^{n_1}$ on $\bc^{n}$ obtained 
from a matrix of exponents $A$ of rank $n_1$, having positive integral  entries and constant row sums as in (iii) above.  
This induces a linear action of $T_1$ on $\Lambda^k(\bc^n)\cong \bc^{{n\choose k}}$ for $k<n$. 
Denote by $G_k(\bc^n)$ the Grassmann variety 
of $k$ dimensional vector subspaces of $\bc^n$. 
The standard $T_1$-action on 
$L_1=\Lambda^k(\bc^n)\setminus\{0\}$  where $\bar{L}_1$ is the tautological line bundle over $\bp(\Lambda^k(\bc^n))$ restricts to a standard $T_1$-action on the 
$L_1|G_k(\bc^n)$ via the Pl\"ucker imbedding  
$G_k(\bc^n)\hookrightarrow \bp(\Lambda^k(\bc^n)).$   Note that 
$\bar{L}_1|G_k(\bc^n)$ is a negative ample line bundle 
over $G_k(\bc^n)$ which generates $Pic(G_k(\bc^n))\cong \bz$.
}
\end{example}

\begin{lemma}\label{proper} 
The orbits of an admissible $\bc_\lambda$-action on $L$ are 
closed and properly imbedded in $L$. 
\end{lemma}
\begin{proof}
Let $p=(p_1,p_2)\in L$.  Let $(z_n)$ be any sequence of 
complex numbers such that $|z_n|\to \infty$.  We shall 
show that $\alpha_\lambda(z_n).p$ has no limit points in $L$. 
Without loss of generality, we may assume that the $z_n$ 
are such that $z_n/|z_n|$ have a limit point $z_0\in \bs^1$. 
By the weak hyperbolicity condition (1), one of the deleted cones 
$z_0C^\circ_i$ is contained in one of the sectors  
$\mathcal{S}_+(\theta):=\{w\in\bc\mid -\theta<\arg(w)<\theta\}\subset \mathcal{R}_+$ or $\mathcal{S}_-(\theta)=-\mathcal{S}_+(\theta)\subset \mathcal{R}_-$ for 
some $\theta$, $0<\theta<\pi/2$.  Say $z_0C^0_i\subset \mathcal{S}_-(\theta)$. Then 
$z_nC^\circ_i\subset \mathcal{S}_-(\theta)$ for all $n$ sufficiently large. 
It follows that $|\exp(z_n\lambda_j)|\to 0$ as $n\to \infty $ for $n_{i-1}<j\leq n_{i-1}+n_i$  (where $n_0=0$). By Lemma \ref{norm} 
we conclude that the sequence $(\alpha_{i}(z_n)(p_i))$ 
does not have a limit in $L_i.$ \end{proof}

\begin{definition}\label{vectorfields}
Given standard  $T_i=(\bc^*)^{n_i}$-actions on the $L_i$, $i=1,2,$  we obtain holomorphic vector fields $v_1,\dots,v_{N}$ on $L=L_1\times L_2$ as follows.  Let $p=(p_1,p_2)\in L.$  Suppose that $1\leq j\leq n_1$. The holomorphic map $\mu_{p_1}:T_1\lr L_1$, $s\mapsto s.p_1,$ induces $d\mu_{p_1}:Lie(T_1)=\bc^{n_1}\lr \mathcal{T}_{p_1}L_1$.  Set  $v_j(p):=(d\mu_p(e_j),0)\in \mathcal{T}_{p_1}L_1\times \mathcal{T}_{p_2}L_2=\mathcal{T}_pL$.  The vector fields 
$v_j, n_1<j\leq N,$ are defined similarly. 
The vector fields $v_j, 1\leq j\leq N,$  are referred to as {\em fundamental vector fields} on $L$. 
\end{definition}

\begin{remark}\label{positivity}{\em
Let $1\leq j\leq n_1$. Consider the differential $d\nu_1:\mathcal{T}_pL\lr \br$  of the norm map $\nu_1:L\lr \br_+$ defined as $q=(q_1,q_2)\mapsto ||q_1||$.   
It is readily verified that, if $ \br_+\epsilon_j$ is not contained in the isotropy at $p_1$, then by standardness of the action,  $d\nu_1(v_j(p))=v_j(p)(\nu_1)=\nu_{j,p_1}'(1)>0.$ 
(Here $\nu_{j,p_1}$ is as in the definition \ref{standard}(ii) of standard action.)
On the other hand, since $\nu_1(s.p)=\nu_1(p)$ for all 
$s\in (\bs^1)^{n_1}=\exp (\sqrt{-1}\br^{n_1})\subset  T_1$ we obtain that
$d\nu_1(\sqrt{-1}v_j(p))=0$.  Thus, for any 
$z\in \bc$, we obtain that $d\nu_1(zv_j(p))=\re(z)\nu_{j,p_1}'(1) $. An entirely analogous statement holds when $n_1<j\leq N$.
}
\end{remark}

Assume that  $\lambda\in \bc^N$ 
yields an admissible imbedding $\alpha\colon\!\bc \lr T$, $\alpha(z)=\exp(z\lambda)$.  We obtain a holomorphic vector field $v_\lambda$ on $L$ where 
\[v_\lambda(p)=\sum_{1\leq j\leq N} \lambda_j v_j(p)\in \mathcal{T}_{p}L.\]

The flow of the vector field $v_\lambda$ yields a holomorphic action of $\bc$ which is just the  restriction 
of the $T$-action to $\bc_\lambda$. 
This $\bc$-action on $L$ is free and the $\bc$-orbits are the same as the leaves of the holomorphic 
foliation defined by the integral curves of the vector field $v_\lambda$.  By Lemma \ref{proper} each leaf is biholomorphic to $\bc$.  It turns out that the leaf space $L/\bc$ is a Haudorff complex analytic manifold and  the projection $L\lr L/\bc_\lambda$ is the projection of a holomorphic  principal bundle with fibre and structure group the additive group $\bc$.    
The underlying differentiable manifold of the leaf space is diffeomorphic to $S(L)=S(L_1)\times S(L_2)$.  
These statements will be 
proved in Theorem \ref{cbundle} below.  We shall denote 
the complex manifold $L/\bc_\lambda$ by $S_\lambda(L).$  The complex structure so obtained on $S(L)$ is referred to as {\it diagonal type}.  

We shall denote by $D(\bar{L})\subset \bar{L}=\bar{L}_1\times \bar{L}_2$ the 
product of the unit disk bundles $D(\bar{L}_i)=\{p\in \bar{L}_i \mid ||p||\leq 1\}\subset \bar{L}_i, i=1,2$.   Also we 
denote by  $\Sigma(\bar{L})\subset \bar{L} $ 
the boundary of $D(\bar{L})$. 
Thus $\Sigma(\bar{L})=
D(\bar{L}_1)\times S(L_2)\cup S(L_1)
\times D(\bar{L}_2).$  Observe that $S(L)=D(\bar{L}_1)\times S(L_2)\cap S(L_1)
\times D(\bar{L}_2)\subset \Sigma(\bar{L})$.

\begin{theorem}\label{cbundle}
With the above notations, suppose that $\alpha_\lambda:\bc\lr T$ 
defines an admissible action of $\bc$ of diagonal type on $L$. Then $L/\bc$ is a (Hausdorff) complex analytic 
manifold and the quotient map $L\lr L/\bc$ is the 
projection of a holomorphic principal $\bc$-bundle.  
Furthermore, 
each $\bc$-orbit meets $S(L)$ transversely at a unique point so that $L/\bc$ is diffeomorphic to $S(L)$. 
\end{theorem}

Proof of the above theorem, which is along the same lines 
as the proof of \cite[Theorem 1]{ln} with suitable  modifications to take care the more general setting we are in, will be based on  
the following two lemmata.  

\begin{lemma}\label{quotient}
Each $\bc_\lambda$-orbit in $L$ meets $S(L)$ at exactly  one point. 
\end{lemma}

\begin{proof} {\it Step 1:}
We first show that each orbit meets $S(L)$ at not more than one point.  
Let $p=(p_1,p_2)\in S(L)$.   Suppose that  $0\neq z\in \bc$ is such that  
$ q:=\alpha_\lambda(z).p=\alpha(z).p\in S(L).$  
This means that, writing  $q=(q_1,q_2)$, we have 
\[q_i=\alpha_i(z)(p_i)=
(\prod_{n_{i-1}<j\leq n_{i-1}+n_i} 
\exp(\lambda_jz)\epsilon_j)p_i, i=1,2,\] (where $n_0=0$). Now  
 $||q_i||=||p_i||=1, i=1,2,$ and $p\neq q$.  Since 
the hermitian metric on $L_1$ is invariant under $(\bs^1)^{n_1}$, we  see that $||p_1||=||q_1||=||(\prod_{1\leq j\leq n_1} (\exp (t_j)\epsilon_j))p_1||$ where $t_j=\re(\lambda_jz)$. Standardness of the $T_1$-action implies that either $\re(\lambda_iz)=0$ for all $i\leq n_1$ or there exist  indices $1\leq i_1<i_2\leq n_1$ such that $ \re(z\lambda_{i_1}).\re(z\lambda_{i_2})<0$. In the latter case 
there exist positive reals $a_1, a_2$ such that  $a_1\re(z\lambda_{i_1})+a_2\re(z\lambda_{i_2})=0$.  Similarly, either $\re(z\lambda_j)=0$ 
for all $n_1<j\leq N$ or   
there exist indices $n_1<j_1<j_2\leq N$ and  positive reals  
$b_1,b_2$ such that $b_1\re(z\lambda_{j_1})+b_2\re(z\lambda_{j_2})=0.$  
Suppose   
$\re(a_1\lambda_{i_1}z+a_2\lambda_{i_2}z)=0=\re(b_1\lambda_{j_1}z+b_2\lambda_{j_2}z)$.  This implies that $a_1\lambda_{i_1}+a_2\lambda_{i_2}=r (b_1\lambda_{j_1}+b_2\lambda_{j_2})$ for 
some positive number $r$. This contradicts the weak hyperbolicity condition (1).  Similarly we obtain a contradiction in the remaining cases as well. 

\noindent 
{\it Step 2:}
Next we show that $\bc p\cap \Sigma(\bar{L})$ is path-connected for any $p\in L$.  We shall write $D_-$ and 
$D_+$ to denote the bounded and unbounded 
components of $L\setminus \Sigma(\bar{L}).$

Without loss of generality, suppose that $p=(p_1,p_2)\in \Sigma(\bar{L})$ and let $q=(q_1,q_2)\in \Sigma(\bar{L})\cap \bc p$ be arbitrary.  Say, $q=\alpha(z_1).p$ with $z_1\neq 0$.  Then 
$r\mapsto \alpha(rz_1).p$ defines a path $\sigma:I\lr \bc p$ with 
end points in $\Sigma(\bar{L})$.  
We modify the path $\sigma$ to obtain a new path 
which lies in $\Sigma(\bar{L}).$   For this purpose choose $z_0\in\bc$ ,  $\arg(z_0)>\frac{\pi}{2}$ such that $z_0C^\circ_1\cup z_0C^\circ_2$ is contained in the left-half space $\mathcal{R}_-=\{z\in \bc\mid \re(z)<0\}$ and $(-z_0)C^\circ_1\cup (-z_0)C^\circ_2$ is contained in the right-half  space $\mathcal{R}_+=\{z\in \bc\mid \re(z)>0\}$.  In particular, $\lim_{r\to \infty}|\exp(rz_0\lambda_j)|=0$ and $\lim_{r\to\infty}
|\exp(-rz_0\lambda_j)|=\infty, \forall j\leq N,$ where $r$ varies in $\br_+$.  By (2), we see that for $i=1,2,$ and any $x_i\in L_i$, $||\alpha_i(rz_0).x_i||\to 0$ and 
$||\alpha_i(-rz_0).x_i||\to \infty$ as $r\to +\infty $ in $\br$.

For any $r\in I$, let $\gamma(r)\in \br$ be least (resp. largest) such that $\alpha(\gamma(r)z_0). \sigma(r)\in \Sigma(\bar{L})$ when $\sigma(r)\in D_+$ (resp. $\sigma(r)\in D_-$). Then $\gamma$ is a well-defined continuous 
function of $r$.  Now $r\mapsto \alpha(\gamma(r)z_0+rz_1).p$ is a path in $\bc p\cap \Sigma(\bar{L})$ 
joining $p$ to $q$. 

\noindent
{\it Step 3: }
To complete the proof, we shall show that, for any $p\in L$, 
there exist points $q'=(q'_1,q'_2), q''=(q_1'',q_2'')\in \bc p\cap \Sigma(\bar{L})$ such that $||q'_1||\leq 1, ||q'_2||=1$ and $||q''_1||=1, ||q''_2||\leq 1$. Then any path in $\bc p\cap \Sigma(\bar{L})$ joining $q'$ and $q''$ must contain a point of $S(L)$.  

Choose $w_k\in \bc^*,1\leq k\leq 4,$ such that the deleted cones $w_1C^\circ_i\subset \mathcal{R}_+, w_2C^\circ_i\subset \mathcal{R}_-$, for $i=1,2$, and,  $w_3C^\circ _1,w_4C^\circ _2\subset R_-$, $w_3C^\circ _2,w_4C^\circ _1\subset \mathcal{R}_+$.  
Then $|\exp(rw_k\lambda_j)|\to 0$ (resp. $\infty$) as $r\to +\infty$ ($r\in \br_+$) 
if $\lambda_j\in C^\circ_i$ and $w_kC^\circ_i\subset \mathcal{R}_-$ (resp. $\mathcal{R}_+$). Now $||\alpha_i(rw_1)p_i||>1, ||\alpha_i(rw_2)p_i||<1, i=1,2$ for  
$r\in\br_+$ sufficiently large.  It follows that any path in $\bc p$ 
joining $\alpha(rw_k)(p), k=1,2,$ must meet $\Sigma(\bar{L})$ for some $r=r_0$. Thus we may 
as well assume that $p\in \Sigma(\bar{L})$.  Suppose that $||p_1||=1, ||p_2||<1$.  For $r>0$ sufficiently large, $||\alpha_1(rw_3).p_1||<1$ and $||\alpha_2(rw_3).p_2||>1$.  Therefore there must exist 
an $r_1$ such that setting $q'_i:=\alpha_i(r_1w_3).p_i$, we have 
$||q'_1||\leq 1$ and $||q'_2||=1$.  Then $q'=(q'_1,q'_2)
\in \bc p\cap \Sigma(\bar{L})$ and $q'':=p$ meet our requirements.  

If $||p_1||<1, ||p_2||=1$, we set $q':=p$ and find a $q''\in \bc p\cap \Sigma(\bar{L})$ 
by the same argument using $w_4$ in the place of $w_3.$  
\end{proof}

\begin{lemma}\label{transverse}
Every $\bc_\lambda$-orbit $\bc p, ~p\in S(L)$, meets $S(L)$ transversally. 
\end{lemma}
\begin{proof} Denote by $\pi:L\lr S(L)$ the projection of the principal $(\bc^*/\bs^1)^2\cong \br_+^2$-bundle.  
Evidently, the inclusion $j:S(L)\hookrightarrow L$ is a 
cross-section and so $L\cong S(L)\times \br^2_+$.  The second projection $\nu:L\lr \br^2_+$ is just the map $L\ni p=(p_1,p_2)\mapsto (\nu_1(p),\nu_2(p))$ where $\nu_i(p)=  ||p_i||\in \br_+$.   
One has therefore an isomorphism 
$\mathcal{T}_p L|_{S(L)}\cong \mathcal{T}_pS(L)\oplus \br^2$, and the corresponding second projection map $\mathcal{T}_p L\lr \br^2$ is the differential of $\nu$.   
Therefore  $\bc p$ is {\it not} transverse to $S(L)$ if and only 
if $av_\lambda(p)\in \mathcal{T}_pS(L)$ for some complex number $a\neq 0$; equivalently, if and 
only if  $d\nu_i(av_\lambda(p))=0, i=1,2,$ for some $a\neq 0$.

By Remark \ref{positivity} we have 
$d\nu_i(av_\lambda(p))=\sum_{1\leq j\leq n_1}d\nu_i(a\lambda_jv_j(p))=
\sum_{1\leq j\leq n_1}\re( a\lambda_j) \nu'_{j,p_1}(1)$.
Similarly, $av_\lambda(p)(\nu_2)=\sum_{n_1<j\leq N}
\re(a\lambda_j)\nu'_{j, p_2}(1)$.   Therefore, $\bc p$ is 
not transverse to $S(L)$ if and only if $
\sum_{1\leq j\leq n_1}
\re(a\lambda_j)r_j=0=
\sum_{n_1<j\leq N}\re(a\lambda_j)s_j$ for some complex number $a\neq 0$ and reals $r_j,s_k\geq 0$ (not all zero).  This means that $\sqrt{-1}\br \subset aC_1^\circ\cap aC_2^\circ$ and 
hence $C_1^\circ\cap C_2^\circ\neq \emptyset $, contradicting the weak hyperbolicity 
condition.  \end{proof}

\noindent
{\it Proof of Theorem \ref{cbundle}:}
 We shall first show that $L/\bc$ is Hausdorff 
by showing that $\pi_\lambda:L\lr S(L)$ which sends 
$p\in L$ to the unique point in $\bc p\cap S(L)$ is continuous.   

Let  
$(p_n)$ be a sequence in $L$ that converges to a point $p_0\in L$. Let $q_n:=\pi_\lambda(p_n)\in S(L)$
and choose $z_n\in \bc$ such that $\alpha(z_n).p_n=q_n$.  
Since $||p_n||, ||q_n||, n\geq 1,$ are bounded, it follows by an argument similar to the proof of 
Lemma \ref{proper} that $(z_n)$ is bounded, and, passing to a subsequence if necessary,  
we may assume that it converges to a $z_0\in \bc$.  
By the continuity of $\bc$-action, $\alpha(z_m).p_n\to \alpha(z_0).p_0$ 
as $m,n\to \infty$.  Therefore $\alpha(z_n).p_n=q_n\to \alpha(z_0).p_0$ 
and $\pi_\lambda(p_0)=q_0$ and so $\pi_\lambda$ is continuous and that the restriction of 
$\pi_\lambda$ to $S(L)$ is a homeomorphism whose 
inverse is the composition $S(L)\hookrightarrow L\lr L/\bc$.

By what has just been shown,  
$L/\bc$ is in fact a Hausdorff manifold and that $\pi_\lambda|_{S(L)}$ is a diffeomorphism.
The orbit space $L/\bc$ has a natural structure of a 
complex analytic space with respect to which the projection $L\lr L/\bc$ is analytic.   Using Lemma \ref{transverse} we see that $L\lr L/\bc$ is a submersion. It follows that $L$ is the total space of a complex analytic principal bundle with fibre and structure 
group $\bc$.  The last statement of the theorem follows from Lemmata \ref{quotient} and \ref{transverse}.
\hfill $\Box$

\begin{remark}\label{egdiag}{\em 
(i) When $X_1$ is a point, one has $X\cong X_2, L\cong \bc^*\times L_2$.  In this case, the orbit space $L/\bc$ 
is readily identified with $L_2/\bz$ where the $\bz$ action is generated by $v\mapsto \prod_{2\leq j\leq N} \exp( 2\pi\sqrt{-1}\lambda_j/\lambda_1)\epsilon_j. v$ where $v\in L_2$.  The projection $L_2\lr S_\lambda(L)$ is a covering projection with deck 
transformation group $\bz$.\\
(ii) When $\lambda$ is of scalar type, the projection $L\lr X$ factors through $S_\lambda(L)$ and yields a complex analytic bundle $S_\lambda(L)\lr X$ with fibre and structure group the elliptic curve 
$(\bc^*\times \bc^*)/\bc$.  
When endowed with diagonal type complex structure  the projection $S_\lambda(L)\lr X$ of the principal 
$\bs^1\times \bs^1$-bundle, which  is smooth, is not 
complex analytic in general.  (Cf. Theorem \ref{transcendence}.)\\
(iii)  When the $X_i$ do not admit any non-trivial $T_i$-
action, we obtain only scalar type complex structures on 
$S(L)$.  For example, this happens when the $X_i$ are 
compact Riemann surfaces of genus at least $2$.\\
(iv) Let $X_1=G_k(\bc^n)$, $X_2=G_l(\bc^m)$.
We start with the example \ref{std}(iv) of standard actions of $T_i,$ corresponding to matrices of exponents $A_i$, constructed on $L_i$ where $\bar{L}_i$ are the negative ample generators of $Pic(X_i)$. 
For any admissible $\bc_\lambda\subset T_1\times T_2$ we obtain a complex structure of diagonal type on 
$S(L)$.\\
(v)  Let $r_i, i=1,2$ be positive integers.  Suppose that $p_i:(\bc^*)^{n_i}\cong \wt{T}_i\lr T_i\cong(\bc^*)^{n_i}$ is the covering projection defined as 
$\prod t_j\epsilon_j\mapsto \prod t_j^{r_i}\epsilon_j$ where $r_i\geq 1$.  Then a standard $T_i$-action on $L_i$ induces a standard $\wt{T}_i$-action. If $\alpha_\lambda:\bc\lr T=T_1\times T_2$ determines an admissible diagonal type action on $L=L_1\times L_2,$  then the lift $\wt{\alpha}_\lambda:\bc\lr \wt{T}$ also determines an admissible 
diagonal type action $\alpha_{\wt{\lambda}}$ where 
$\wt{\lambda}_j=(1/r_1)\lambda_j, 1\leq j\leq n_1,$ 
and $\wt{\lambda}_j=(1/r_2)\lambda_j, n_1<j\leq N$. 
Indeed the resulting $\bc$ action is the `same' and so 
$S_\lambda(L)=S_{\wt{\lambda}}(L)$. In particular, if 
$p_i':\wt{T}_i\lr T_i', i=1,2,$ is another pair of such coverings and  
if $\alpha_\lambda:\bc\lr T$ and $\alpha_{\lambda'}:\bc\lr T'$ 
define admissible diagonal type actions on $L$ such that $\alpha_{\wt{\lambda}}:\bc\lr \wt{T}$ is a common lift of both $\alpha_{\lambda}$ and $\alpha_{\lambda'}$, then $S_\lambda=S_{\wt{\lambda}}=S_{\lambda'}$.
}
\end{remark}

We conclude this section with the following observation.

\begin{theorem}  \label{symplectic}
Suppose that $H^1(X_1;\br)=0$ and that $c_1(\bar{L}_1)\in H^2(X_1;\br)$ is non-zero.  Then $S(L)$ is not symplectic and hence 
non-K\"ahler with respect to {\em any} complex structure.  
\end{theorem}

\begin{proof}
In the Leray-Serre spectral sequence over $\br$ for the 
$\bs^1$-bundle with projection $q:S(L_1)\lr X$ the differential $d:E^{0,1}_2\cong H^1(\bs^1;\br)\cong \br \lr E^{2,0}_2=H^2(X_1;\br)$ is non-zero. 
It follows that $E^{0,1}_3=E^{0,1}_\infty=0$. Since $H^1(X_1;\br)=0$, we see that $H^1(S(L_1);\br)=0$. 
Hence, by the K\"unneth formula, $H^2(S(L);\br)=H^2(S(L_1);\br)\oplus H^2(S(L_2);\br)$. 

Let $u_i\in H^2(S(L_i);\br),i=1,2,$ be arbitrary.  Since $\dim S(L_i)$ is odd for $i=1,2$, 
$u_1^ru_2^s=0$ for any $r, s\geq 0$ such that $r+s=n$, where $2n:=\dim_\br S(L)$.  Hence 
$\omega^n=0$ for any $\omega\in H^2(S(L);\br)$.  
\end{proof}



\section{Complex structures of linear type}
Let $X_i=G_i/P_i, i=1,2$, where the $G_i$  are simply-connected complex simple linear algebraic groups, $P_i$ any maximal parabolic subgroup,  and  $\bar{L}_i$ the negative ample generator of the Picard group of $X_i$.
We endow $\bar{L}_i$ with a hermitian metric invariant under a suitable maximal compact subgroup $H_i\subset G_i$. Let $L=L_1\times L_2$ and let $S(L)$ be product $S(L_1)\times S(L_2)$ where $S(L_i)\subset L_i$ is the unit circle bundle over $X_i, i=1,2$.  
It can been seen that $S(L_i)$ is simply-connected. Indeed it is a homogeneous space $H_i/Q_i$  where $Q_i$ is connected and is the semi simple part of the centralizer of a circle subgroup contained in $H_i$  (see \cite{rs}, \cite{sankaran}). By a classical result of  H.-C. Wang \cite{wang}, it follows that  $S(L)$ admits complex structures invariant under the action of $H_1\times H_2$.  The complex structures considered by Wang are the same as those of scalar type considered in \S2.  
The  $H:=H_1\times H_2$-action does not   
preserve the complex analytic structure when $S(L)$ is endowed with the more general diagonal type complex structures.   

In this section we shall construct a complex structure on $S(L)$ which will be referred to as 
{\it linear type}. Our construction will be more general 
in that {\it we assume only that $G_i$ is any simply connected semi simple Lie group and $P_i\subset G_i$ any parabolic subgroup.}

The first step towards construction of 
linear type complex structure on $S(L)$ is to produce a standard action of a torus $T'_i\subset G_i$ on $L_i\lr G_i/P_i$.  The following consideration shows that there can be no such action for any torus of the semi simple group $G_i$. 

Suppose that $\bar{L}\lr G/P$ is  a $G$-equivariant line bundle where $G$ is a simply-connected semi simple complex linear algebraic group and $P$ {\it any} parabolic subgroup of $G$. We assume that $G$ action on $G/P$ is almost effective, as otherwise $G/P=G'/P'$ where $G'$ is proper factor of $G$ and $P'=P\cap G'$.   
(Almost effective action means that the subgroup that fixes every element of $G/P$ is finite.) Now the subgroup 
of $G$ which fixes every element of $G/P$ is readily seen 
to be equal to the centre $Z(G)$ of $G$. 
Let $T'$ be any torus of $G$. We claim  that the $T'$-action on $L$ is {\it not} $d$-standard for any $d\geq 1$ (with respect to any 
isomorphism $T'\cong (\bc^*)^k$, where $k\leq l=rank(G)$). 
If the $T'$-action were $d$-standard, then $T'$ would contain a subgroup $\Delta\cong\bc^*$ whose restricted action is as described in Definition \ref{standard}(i).  Since the $G$-action commutes with that 
of the structure group $\bc^*$ of $\bar{L}$, it follows that $z.g(v)=g.z(v)$ for all $v\in \bar{L}, z\in\Delta,g\in G$.  Since the $G$-action on $L$ is almost effective, we see $g^{-1}zg=\zeta z $ where $\zeta\in Z(G)$, the centre of $G$, which is a finite group.  This implies that $\Delta/(\Delta\cap Z(G))$ is contained in the centre of $G/Z(G)$ contradicting our hypothesis that $G$ is semi simple.  

We shall show in Proposition \ref{homogstd} that, when $\bar{L}$ is a line bundle associated to a negative dominant integral weight, it is possible to extend the $G$ action on $\bar{L}$ and on $G/P$ to a larger group $\wt{G}$ which is reductive such that the bundle $\bar{L}\to G/P$ is $\wt{G}$-equivariant and the action of a maximal torus $\wt{T}$ of $\wt{G}$ on $L$ is $d$-standard for a suitable $d\geq 1$.  

In order to construct linear type complex structure on 
$S(L)$, we need to assume that $\bar{L}_i, i=1,2,$ is a negative ample line bundle over $G_i/P_i$. This assumption
allows us to view $\bar{L}_i$ as the restriction of 
the tautological bundle over a projective space $\bp^{N_i}$ to $G_i/P_i$ via an imbedding $G_i/P_i\hookrightarrow \bp^{N_i}$ defined by the very ample line bundle $\bar{L}^\vee$. 
As this fact will be exploited in our construction of linear type complex structure, it fails when line bundle $\bar{L}_i$ is not negative ample. 

We briefly recall some basic facts and notions about the representation theory of $G$, referring the reader to \cite{humphreys} for details. 

Let $G$ be a semi simple, simply-connected complex linear algebraic group.  Let $T$ be a maximal torus and let $B$ be a  
Borel subgroup $B$ containing $T$.  Let $l=\dim T$ be the rank of $G$. Denote by $R(G)$---or more briefly $R$---the set of roots,  by $R^+$ the positive roots, by $\Lambda$ the weight lattice and by $Q\subset \Lambda$ the root lattice determined by $T\subset B\subset G$.  We shall denote the set of coroots by $R^\vee$. Since $G$ is assumed to be simply connected, $\Lambda=\chi(T)$, the group of characters $T\lr \bc^*$ of $T$.  Since $B=T. B_u$, $B_u$ being the unipotent, 
every character of $T$ extends uniquely to a (algebraic) character of $B$ and we have $\chi(T)=\chi(B)$. 
Let $\Phi^+\subset R^+$ denote the set of simple positive roots and let $\Lambda^+\subset \Lambda$ denote the dominant (integral) weights.   We shall denote by $W$ the Weyl group of $G$ with respect to $T$. It is 
generated by the set  $S$ of  the fundamental reflections 
$s_\alpha, \alpha\in \Phi^+$.  $(W,S)$ is a finite Coxeter group whose longest element will be denoted $w_0$.  

For $\omega\in \Lambda^+$, $V(\omega)$ denotes the finite dimensional irreducible highest weight $G$-module with highest weight $\omega$. Also, for any  $\omega\in \Lambda,$ one has a $G$-equivariant line bundle $\bar{L}_\omega\lr G/B$ whose total space is $G\times_B \bc_{-\omega}$ where $\bc_{-\omega}$ is the $1$-dimensional $B$-module with character $-\omega:B\lr \bc^*$.  If $\omega$ is dominant, then $H^0(G/B,L_\omega)^\vee=V(\omega)$ as $G$-module.  If $u_\omega\in V(\omega)$ is a highest weight vector, then $P_\omega$, the subgroup of $G$ which stabilizes the $1$-dimensional vector space $\bc u_\omega$ is a parabolic subgroup that contains $B$ and $\bar{L}_\omega$ is isomorphic to the pull-back of a line bundle, again denoted $\bar{L}_\omega$ over $G/P$ where $P$ is any parabolic subgroup such that $B\subset P\subset P_\omega$.  Every parabolic subgroup that contains $B$ arises as $P_\omega$ for some $\omega\in \Lambda^+$.
Moreover, $\bar{L}_\omega\lr G/P_\omega$ is (very) ample. If $\omega$ is a positive multiple of a fundamental weight $\varpi$, then $P_\omega$ is a maximal parabolic  which corresponds to `omitting' $\varpi$.  
  
Let $\omega\in \Lambda^+$ and  
let $\Lambda(\omega)\subset \Lambda$ denote the  set of all weights of $V(\omega).$  If $\mu\in \Lambda(\omega)$, we denote the multiplicity of $\mu$ in $V(\omega)$ by $m_\mu$; thus $m_\mu=\dim V_\mu(\omega)$, where $V_\mu(\omega)$ is the $\mu$-weight space $\{v\in V(\omega)\mid t.v=\mu(t)v~\forall 
t\in T\}$.  The set  
$\Lambda(\omega)$ is stable under the action of $W$.   We put a hermitian inner product on $V(\omega)$ with respect to which the decomposition $V(\omega)=\oplus_{\mu\in \Lambda(\omega)} V_\mu(\omega)$ is orthogonal.  Such an hermitian product is invariant under the compact torus $K\subset T$.  Indeed, without loss of generality we may assume that the inner product is invariant under a maximal compact subgroup of $G$ that contains $K$. 

Let $\varpi_1,\dots, \varpi_l$ be the fundamental weights. 
Consider the homomorphism $\psi:T\lr (\bc^*)^l$ of algebraic groups defined as 
$t\mapsto (\varpi_1(t), \dots, \varpi_l(t))$.  It is an isomorphism since $\varpi_1,\dots,\varpi_l$ is a $\bz$-basis for $\chi(T)$.   We shall identify $T$ with $(\bc^*)^l$ via $\psi$. Let $\omega\in \Lambda^+$.   It is not difficult to see that the $T$-action on $V(\omega)\setminus{0}\lr P(V(\omega))$ is {\it not} standard since  $w_0(\omega)\in \Lambda(\omega)$ is {\it negative} dominant, i.e., $-w_0(\omega)\in \Lambda^+$. 
Write $\mu=\sum_{1\leq j\leq l}a_{\mu, j}\varpi_j$ for $\mu\in \Lambda(\omega)$ so that $\mu(t)=\prod_{1\leq j\leq l}t_j^{a_{\mu,j}}$ where $t=(t_1,\dots,t_l)\in T$.  If  $v\in V_\mu(\omega)$, then $t.v=
\prod t_j^{a_{\mu,j}}.v$.   Set 
$d':=1+\sum |a_{\mu,j}|$ where the sum is over $\mu\in \Lambda(\omega), 1\leq j\leq l$. The group $T':= T\times \bc^*$ acts on $V(\omega)$ where the last factor acts via the covering projection $\bc^*\lr \bc^*$, $z\mapsto z^{-d'},$ where the target $\bc^*$ acts as scalar multiplication.  Thus $(t,z).v=\mu(t)z^{-d'}v$ where $v\in V_\mu(\omega), (t,z)\in T'$.  
Now consider the $(l+1)$-fold covering projection $\wt{T}:=(\bc^*)^{l+1}\lr T',$   defined as $ (t_1,\dots,t_{l+1})\mapsto (t_{l+1}^{-1}t_1,\dots, t_{l+1}^{-1}t_l,\prod_{1\leq j\leq l+1} t^{-1}_j)$.  
The torus $\wt{T}$ acts on the principal $\bc^*$-bundle $V(\omega)\setminus\{0\}\lr \bp(V(\omega))$ via the above surjection.    

Denote by $\wt{\epsilon}_j:\bc^*\lr \wt{T}$ the $j$th coordinate imbedding.
For any $\mu\in \Lambda(\omega)$, 
and any $v\in V_\mu(\omega)$, we have  $z\wt{\epsilon}_{l+1}.v= z^{d'}\prod_{1\leq j\leq l}z^{-a_{\mu,j}}v=z^{d'-\sum a_{\mu,j}}v$, and, 
when $j\leq l$, we have $z\wt{\epsilon}_j.v=z^{d'+a_{\mu,j}}v$.  Also, if $z=(z_0,\dots,z_0)\in \wt{T}$, then 
$z.v=z_0^{(l+1)d'}v$.
Observe that the exponent of $z$ that occurs in the above formula for $z\wt{\epsilon_j}.v$ is positive for $1\leq j\leq l+1$ by our choice of $d'$.     We shall denote this exponent by $d_{\mu,j}$, that is, 
\[
d_{\mu,j}=\left\{
\begin{array}{lr}
                     d'+a_{\mu,j}, & 1\leq j\leq l,\\
                     d'-\sum_{1\leq i\leq l} a_{\mu,i}, & j=l+1,\\
                    \end{array} \right . \eqno{(3)}
\]                    
where $\mu=\sum_{1\leq j\leq l}a_{\mu,j}\varpi_j\in \Lambda(\omega)$.

Next note that the compact torus $\wt{K}:=K\times \bs^1\subset T\times  \bc^*$ preserves the hermitian 
product on $V(\omega)$ and hence the (induced) hermitian metric on the tautological line bundle 
over $\bp(V(\omega))$.   From the explicit description of the action just given, it is clear that 
conditions (i) and (ii) of Definition \ref{standard} hold. Thus we have extended the $T$-action to an action of $\wt{T}$-action which is standard.    
We are ready to prove 

\begin{proposition}  \label{homogstd} We keep the above notations.
Let $\omega \in \Lambda^+$ be any dominant weight of $G$ and let $P$ be any parabolic subgroup such that 
$B\subset P\subset P_\omega.$ 
Then the $T$-action can be extended to a $d$-standard action of $\wt{T}:=T\times \bc^*$ on  
$L_{-\omega}\lr G/P$ where $d=d'(l+1)$. 
\end{proposition}
\begin{proof}  
First assume that $P=P_\omega$.
Since $\bar{L}_\omega$ is a very ample line bundle over $G/P_\omega$, one has a $G$-equivariant embedding 
$G/P_\omega\lr \bp(V(\omega))$ where $V(\omega)=H^0(G/P_{\omega}, \bar{L}_\omega)^\vee.$   By our discussion above, the $T$-action on the tautological bundle over the projective space $\bp(V(\omega))$ has been extended 
to a $d$-standard action of $\wt{T}$ for an appropriate $d>1$.  The tautological bundle over $\bp(V(\omega))$ restricts to $L_{-\omega}$ on $G/P_\omega$. Clearly the $L_{-\omega}$ is $\wt{T}$-invariant.   Put any $\wt{K}$-invariant 
hermitian metric on $V(\omega)$ where $\wt{K}$ denotes the maximal compact subgroup of  $\wt{T}$.
As observed above, $z\wt{\epsilon}_j.v=z^{d_{\mu,j}}v$ 
where $d_{\mu,j}>0$ for $v\in V_\mu(\omega)$, it 
follows that condition (ii) of Definition \ref{standard} holds. Therefore the $\wt{T}$-action on $L_{-\omega}$ is 
$d$-standard.  

Now let $P$ be any parabolic subgroup as in the proposition. 
One has a $\wt{T}$-equivariant morphism $G/P\lr G/P_\omega$ under which $\bar{L}_{-\omega}\lr G/P_\omega$ pulls back to $\bar{L}_{-\omega}\lr G/P$. 
In view of Remark \ref{std}(ii), it follows that the $\wt{T}$-action on $L_{-\omega}\lr G/P$ is $d$-standard.  
\end{proof}

Let $\pi:\wt{G}=G\times \bc^*\lr G\times \bc^*$ be the $(l+1)$-fold covering obtained from the 
$(l+1)$-fold covering of the last factor and identity on the first.  The maximal torus $\pi^{-1}(T\times\bc^*)$ of 
$\wt{G}$  can be identified with $\wt{T}$.  With respect to an appropriate choice of identification $\wt{T}\cong (\bc^*)^{l+1}$,  we see that the 
action of  $G$ on $L_{-\omega}\lr G/P_\omega$ extends to $\wt{G}$ in such a manner that the $\wt{T}$-action is $d$-standard where $d=(l+1)d'$ as above.  Since $G/P_\omega=\wt{G}/\wt{P}_\omega$, where $\wt{P}=\pi^{-1}( P_\omega\times \bc^*)$,  the $\bc^*$-bundle $L_{-\omega}\lr \wt{G}/\wt{P}_\omega$ is $\wt{G}$-equivariant.  The parabolic subgroup $\wt{P}_\omega$ contains the Borel subgroup $\wt{B}:= \pi^{-1}(B\times \bc^*)$.  We shall refer to $\bar{L}_{-\omega}\lr \wt{G}/\wt{P}_\omega$ as a $d$-{\it standard $\wt{G}$-homogeneous line bundle}.

Let $i=1,2$.  We shall write $L_i, P_i$ to abbreviate $L_{-\omega_i}, P_{\omega_i}$, etc. Note that $\bar{L}_i$ is a negative ample line over $X_i:=G_i/P_i=\wt{G}_i/\wt{P}_i$ and is $d_i$-standard $\wt{G}_i$-homogeneous. Let $\wt{G}=\wt{G}_1\times \wt{G}_2$, $(\bc^*)^N\cong\wt{T}=\wt{T}_1\times \wt{T}_2$ where $N=rank( \wt{G})=n_1+n_2$ with $n_i:=l_i+1$   and $\wt{B}=\wt{B}_1\times \wt{B}_2$.    

Let $\lambda\in Lie(\wt{B})$ and let $\lambda=\lambda_s+\lambda_u$ be its Jordan decomposition, where  $\lambda_s=(\lambda_1,\dots,\lambda_N)\in \bc^N=Lie(\wt{T})$ satisfies the weak hyperbolicity condition (1) of type $(n_1,n_2)$ and $\lambda_u\in Lie(\wt{B}_u)$, the 
 Lie algebra of the unipotent radical $\wt{B}_u$ of $\wt{B}$.  Thus $[\lambda_u,\lambda_s]=0$ in $Lie(\wt{B})$.  The analytic imbedding $\alpha_\lambda:\bc\lr \wt{B}$ where $\alpha_\lambda(z)=\exp(z\lambda)=\exp(z\lambda_s).\exp(z\lambda_u)$ defines an action, again denoted $\alpha_\lambda$, of $\bc$ on $L:=L_1\times L_2$ and an action 
 $\wt{\alpha}_\lambda$ on  $V(\omega_1)\times V(\omega_2)$.    Denote by $\bc_\lambda$ the image 
 $\alpha_\lambda(\bc)\subset \wt{B}$. 
 We shall now give an explicit description of these  
 actions.  Let $v_i\in V(\omega_i)$ and write $v_i=\sum_{\mu\in \Lambda(\omega_i)}v_{\mu}$ where 
 $v_{\mu}\in V_\mu(\omega_i)$.   Set  
\[\lambda_\mu:=\sum \lambda_j d_{\mu,j} \eqno{(4)}\]  
where the sum ranges over ${n_{i-1}< j\leq n_{i-1}+n_i}$ with $n_0=0$.  
Then $\wt{\alpha}_{\lambda_s}(z)(v_1,v_2)=(u_1,u_2)$ where 
 \[u_i=\sum_{\mu\in \Lambda(\omega_i)}\prod_{j}\exp(z\lambda_j)\wt{\epsilon}_j.v_\mu=\sum_\mu \prod_j 
 (\exp(z\lambda_jd_{\mu,j})v_\mu 
 =\sum_\mu \exp(z\lambda_{\mu})v_\mu. \eqno{(5)}
 \]
where the product is over $j$ such that ${n_{i-1}<j\leq n_{i-1}+n_i}$.
 
The $\bc$-action $\alpha_{\lambda_s}$ on $L$ is just the restriction to  $L\subset V(\omega_1)\times V(\omega_2)$ of the $\bc$-action $\wt{\alpha}_{\lambda_s}$.  Since the $\lambda_\mu$ are all {\it positive} linear combination of the $\lambda_j$, the action of $\bc$ on  
$V(\omega_1,\omega_2):= (V(\omega_1)\setminus\{0\})\times (V(\omega_2)\setminus\{0\})$, the total space of the product of tautological bundles, is admissible.
 
Fix a basis for $V(\omega_i)$ consisting of weight vectors so that $GL(V(\omega_i))$ is identified with invertible $r_i\times r_i$-matrices,  where $r_i:=\dim V(\omega_i)$.
 Note that the action of the diagonal subgroup of $GL(V(\omega_i))$  on $V(\omega_i)\setminus \{0\}$ is standard and that $\wt{T}$ is mapped into $D$, the diagonal subgroup of $GL(V(\omega_1))\times GL(V(\omega_2))$. 
We put a hermitian metric on $V(\omega_1)\times V(\omega_2)$ which is invariant under the compact torus $(\bs^1)^{r_1+r_2}\subset D$.
    Considered as a subgroup of $GL(V(\omega_1))\times GL(V(\omega_2))$, the $\bc$-action $\wt{\alpha}_{\lambda_s}$ on $V(\omega_1,\omega_2)$ is the {\it same} as that considered by Loeb-Nicolau corresponding to $\lambda_s(\omega_1,\omega_2):= 
    (\lambda_{\mu},\lambda_{\nu})_{\mu\in \Lambda(\omega_1),\nu\in \Lambda(\omega_2)}\in Lie(D)=\bc^{r_1}\times \bc^{r_2}$, where it is understood that each $\lambda_{\mu}$ occurs as many times as $\dim V_\mu(\omega_1), \mu\in \Lambda(\omega_1),$ and similarly for $\lambda_\nu,\nu\in \Lambda(\omega_2)$.   

\noindent 
{\bf Observation:} The $\lambda_s(\omega_1,\omega_2)$ satisfy the weak hyperbolicity condition of type $(r_1,r_2)$ since the $\lambda_\mu$ are {\it positive} integral linear combinations of the $\lambda_j$.

 The differential of the Lie group homomorphism $\wt{G}_1\times \wt{G}_2\lr GL(V(\omega_1))\times GL(V(\omega_2))$ maps $\lambda_s$ to the diagonal matrix $diag(\lambda_s(\omega_1,\omega_2))$ and 
$\lambda_u$ to a nilpotent matrix 
$\lambda_u(\omega_1,\omega_2)$ which commutes with 
$\lambda_s(\omega_1,\omega_2)$. Indeed 
$\lambda(\omega_1,\omega_2):=\lambda_s(\omega_1,\omega_2)+\lambda_u(\omega_1,\omega_2)$ has a block decomposition compatible with weight-decomposition of $V(\omega_1)\times V(\omega_2)$ where the $\mu$-th block is 
$\lambda_\mu I_{m(\mu)}+A_\mu$, where $A_\mu$ is nilpotent and $I_{m(\mu)}$ is the identity matrix of size 
$m(\mu)$, the multiplicity of $\mu\in \Lambda(\omega_i), i=1,2$.

Recall that, for the $\bc$-action $\wt{\alpha}_{\lambda}$ on $V(\omega_1,\omega_2),$  
the orbit space $V(\omega_1,\omega_2)/\bc=:S_\lambda(\omega_1,\omega_2)$ is a complex manifold diffeomorphic to the product of spheres $\bs^{2r_1-1}\times \bs^{2r_2-1}$ by \cite[Theorem 1]{ln}.   Indeed, the canonical projection $V(\omega_1,\omega_2)\lr S_\lambda(\omega_1,\omega_2)$ is the  
projection of a holomorphic principal bundle with fibre and structure group $\bc$.

\begin{theorem}\label{lineartype}  We keep the above notations. Let $\bar{L}_i=\bar{L}_{-\omega_i}$ be a  
$d_i$-standard $\wt{G}_i$-homogeneous line bundle over $X_i=\wt{G}_i/\wt{P}_i, P_i=P_{\omega_i}$ and let $L=L_1\times L_2$.    Suppose that $\lambda=\lambda_s+\lambda_u\in Lie(\wt{B})$ where 
$\lambda_s\in Lie(\wt{T})=\bc^N$ satisfies the weak hyperbolicity condition of type $(n_1,n_2)$.  (See Equation (1), \S2.)\\
(i) The orbit space, denoted $L/\bc_\lambda$, of the $\bc$-action on $L$ defined by $\lambda$ is a Hausdorff complex manifold and the canonical projection $L\lr L/\bc_\lambda$ is the projection of a principal $\bc$-bundle.    Furthermore, 
$L/\bc_\lambda$ is analytically isomorphic to $L/\bc_{\lambda_\varepsilon}$ where $\lambda_\varepsilon:=
Ad(t_\varepsilon)(\lambda)$ and $t_\varepsilon\in \wt{T}$ is such that $\gamma(t_\varepsilon)=\varepsilon$ for all 
$\gamma\in \Phi^+$.\\
(ii)  If $|\varepsilon|$ is sufficiently small, then each orbit of $\bc_{\lambda_\varepsilon}$ on $L$ meets $S(L)$ transversally at a unique point. In particular, the restriction of the projection $L\lr L/\bc_{\lambda_\varepsilon}$ to $S(L)\subset L$ is a diffeomorphism.
\end{theorem} 
\begin{proof}
When $\lambda_u=0$, the theorem is a special case of 
Theorem \ref{cbundle}.  So assume $\lambda_u\neq 0$.

Since the $\bc$-action $\alpha_{\lambda_\varepsilon}$ is conjugate by the analytic 
automorphism $t_\varepsilon:L\lr L$ to $\alpha_\lambda$, we see that $L/\bc_\lambda\cong L/\bc_{\lambda_\epsilon}$ as a complex analytic space. Thus, it is enough to prove the theorem for $|\varepsilon|>0$ sufficiently small.   

Consider the projective embedding  $\phi_i':X_i=G_i/P_i \hookrightarrow \bp(V(\omega_i))$ defined by the ample line bundle $\bar{L}_i^\vee$.  The circle-bundle $S(L_i)\lr X_i$ is just the restriction to $X_i$ of the circle-bundle associated to the tautological bundle over $\bp(V(\omega_i))$. Thus $\phi_i'$ yields an imbedding 
$\phi_i:S(L_i)\lr \bs^{2r_i-1}$.   Let $\phi:S(L)\lr 
S(V(\omega_1,\omega_2))=\bs^{2r_1-1}\times \bs^{2r_2-1}$ be the product $\phi_1\times \phi_2$.   

Set $\lambda_{u,\varepsilon}=Ad(t_\varepsilon)\lambda_u$ so that $\lambda_{\varepsilon}=\lambda_s+\lambda_{u,\varepsilon}$. Note that, if $\beta=\sum_{\gamma\in \Phi^+}k_{\beta,\gamma}\gamma$ where $\beta\in R^+$, then $Ad(t_\varepsilon)X_\beta=\varepsilon^{|\beta|}X_\beta$ where 
$|\beta|=\sum k_{\beta,\gamma}\geq 1$.   ( Here $X_\beta\in Lie(B_u)$ denotes a weight vector 
of weight $\beta$. )
This implies that $\lambda_\varepsilon\to \lambda_s$ as $\varepsilon\to 0$, and,  furthermore,  
$\lambda_\varepsilon(\omega_1,\omega_2)\to \lambda_s(\omega_1,\omega_2)$ as $\varepsilon\to 0$.  
By \cite[Theorem 1]{ln}, for $|\varepsilon|$ sufficiently small, each $\bc$-orbit for the $\wt{\alpha}_{\lambda_\varepsilon}$-action on $V(\omega_1,\omega_2)$ 
is closed and properly imbedded in $L(\omega_1,\omega_2)$ and intersects $\bs^{2r_1-1}
\times \bs^{2r_2-1}$ at a unique point.  In particular, each orbit of the $\bc$-action corresponding to $\lambda_\varepsilon$ meets $S(L)\subset\bs^{2r_1-1}\times \bs^{2r_2-1}$ at a {\it unique} point when 
$|\varepsilon|>0$ is sufficiently small.  

Consider the map 
$\pi_{\lambda_\varepsilon}:L\lr S(L)$ which maps each 
$\alpha_{\lambda_\varepsilon}$ orbit to the unique point where it meets $S(L)$.  This is just the restriction of 
$V(\omega_1,\omega_2)\lr \bs^{2r_1-1}\times \bs^{2r_2-1}$ and hence continuous.  It follows that the 
orbit space $L/\bc_{\lambda_\varepsilon}$ is Hausdorff and that the map 
$\bar{\pi}_{\lambda_\varepsilon}:L/\bc\lr S(L)$ induced by $\pi_{\lambda_\varepsilon}$ is a homeomorphism, whose inverse is just the composition $S(L)\hookrightarrow L\lr L/\bc$. 
Since each $\bc$-orbit for $\alpha_{\lambda_s}$-action meets $S(L)$ transversely by Lemma \ref{transverse}, and since $S(L)$ is compact, the same is true for the $\alpha_{\lambda_\varepsilon}$-action provided 
$|\varepsilon|$ is sufficiently small.  For such an 
$\varepsilon$, the $\pi_{\lambda_\varepsilon}$ is a submersion and $\bar{\pi}_{\lambda_\varepsilon}$ is a diffeomorphism.  
The orbit space $L/\bc_{\lambda_\varepsilon}$ has a natural structure of a 
complex analytic space with respect to which $\pi_{\lambda_\varepsilon}$ is analytic. We have shown above that $L/\bc_{\lambda_\varepsilon}$ is a Hausdorff manifold and that   
$\pi_{\lambda_\varepsilon}$ is a submersion.  It follows that $\pi_{\lambda_\varepsilon}$ is the projection of a principal complex analytic bundle with fibre and structure group $\bc$.  \end{proof} 

Let $P'=P'_1\times P'_2$ be any parabolic subgroup of $G=G_1\times G_2$ such that $B\subset P'\subset P$, where $P=P_1\times P_2$ is an Theorem \ref{lineartype}.  
Let $L'=L'_1\times L'_2$ where $\bar{L}'_i$ is the line bundle over $X_i':=G_i/P_i'$ associated to $-\omega_i$, where $\omega_i$ is a dominant integral weight of $T_i$. 
Then $L'\lr X'$ is a $\wt{B}$-equivariant line bundle 
and so one obtains an action of $\bc_\lambda$ on $L'$ via restriction for any 
$\lambda\in Lie(\wt{B})$. Moreover, $L'$ is equivariantly isomorphic to the pull back of $L$ via 
the natural projection $X'\lr X$ where $L, X=X_1\times X_2$ are 
as in Theorem \ref{lineartype}.  In particular, assuming that $\lambda_s$ satisfies the weak hyperbolicity condition, one has a $\bc_\lambda$-equivariant 
projection $L'\lr L$.  
If $\lambda_u=0$, then, by Proposition \ref{homogstd} 
and Theorem \ref{cbundle},  $L'/\bc_\lambda$ is compact Hausdorff complex manifold and we have the following 
commuting diagram:  
\[
\begin{array}{ccc}
L' & \lr & L\\
\downarrow & &\downarrow \\
L'/\bc_\lambda &\lr & L/\bc_\lambda.\\
\end{array}\]   
However, it is not clear to us whether the orbit space $L'/\bc_\lambda$ is a Hausdorff manifold when $\lambda_u\neq 0$. 

\begin{definition} \label{lineartypedef}With notations as in Theorem \ref{lineartype}, if $\lambda_s$ is weakly hyperbolic, we shall refer to the analytic homomorphism $\alpha_\lambda:\bc\lr \wt{B}$ as {\em admissible} and the action $\alpha_\lambda$ of $\bc$  
on $L$ as {\em linear type}.  
In this case, complex structure on the manifold $S(L)$ induced from $L/\bc_{\lambda_\varepsilon}\cong L/\bc_{\lambda}$ will be said to be of {\em linear type} and the resulting complex manifold will be denoted $S_{\lambda}(L)$.   
\end{definition}

 We conclude this section with the following remarks.

\begin{remark}\label{comparision}{\em  (i) Loeb and Nicolau \cite{ln} consider more general $\bc$-actions on $\bc^m\times \bc^n$ in which the corresponding vector field is allowed to have higher order resonant terms.  In our setup we have only to consider linear 
actions---the corresponding vector fields can at most have 
terms corresponding to resonant relations of the form 
$``\lambda_i=\lambda_j".$ \\
(ii) One has a commuting diagram\\
$$\begin{array}{ccc}
L &\hookrightarrow  & V(\omega_1,\omega_2)\\
\pi_\lambda\downarrow & & \downarrow\pi_{\lambda(\omega_1,\omega_2)}\\
S_\lambda(L)&\hookrightarrow &\bs^{2r_1-1}\times \bs^{2r_2-1}\\
 \end{array}$$
 in which the horizontal maps are holomorphic and the vertical maps, projections of holomorphic principal $\bc$-bundles.\\
 (iii) We do not know if any diagonal 
 type action of $\bc$ on $L$ (with $L\lr X$ as in Definition \ref{lineartypedef}) in the sense of Definition \ref{diagonaltype} 
 is conjugate by an element of $\wt{G}$ to a linear 
 type action $\alpha_\lambda$ with unipotent part 
 $\lambda_u$ equal to zero.    
 }
\end{remark}

\section{Cohomology of $S_\lambda(L)$}

Let $p:L\lr X$ be a holomorphic principal $\bc^*$-bundle over a compact connected complex 
manifold. Denote by $\bar{p}:\bar{L}\lr X$ the 
associated line bundle (i.e., vector bundle of rank $1$).  
We identify $X$ with the image of the zero 
cross section in $\bar{L}$ and $L$ with $\bar{L}\setminus X$.   

We denote the structure sheaf of a complex analytic 
space $Y$ by $\co_Y$, or more briefly by $\co$ when $Y$ is clear from the context.  Recall that a compact subset $A\subset Y$ is called {\it Stein compact} if every 
neighbourhood of $A$ contains a Stein open subset $U$ such that $A\subset U$. 

\begin{lemma}  \label{cohenmacaulay} 
Suppose that $f:\bar{L}\lr Y$ is a 
complex analytic map where $Y$ is a Stein space 
and is Cohen-Macaulay. Suppose that  $f(X)=:A\subset Y$ is Stein compact and that $f|_L:L\lr Y\setminus A$ is a biholomorphism.  Then 
$H^q(L;\co_L)=0$ if $0<q<\dim Y-\dim A-1$, or if $q=\dim L$.  
Furthermore, $H^0(L;\mathcal{O}_L)\cong H^0(Y;\mathcal{O}_{Y})$ and the topological vector spaces 
$H^q(L;\co_L)$ are separated Fr\'echet-Schwartz 
spaces for all $q$.   
\end{lemma}

\begin{proof}  
First note that $H^q(Y;\co_{Y})=0$ unless 
$q=0$.   By \cite[Ch. II, Theorem 3.6, Corollary 3.9]{bs}, 
the restriction map 
$H^q(Y;\co)\lr H^q(Y\setminus A;\co)$ 
is an isomorphism for $0\leq q
<\textrm{depth}_A\co_{Y}=\dim X-\dim A$ the last equality in view of our hypothesis that 
$Y$ is Cohen-Macaulay.  In view of \cite[Ch. I, Theorem 2.19]{bs}, 
our hypothesis that $A\subset Y$ is Stein compact implies that $H^i(Y\setminus A;\co)\cong H^i(L,\co_L)$ is separated and Fr\'echet-Schwartz. As for any open connected complex manifold, 
$H^{\dim L}(L;\co)=0$. 
\end{proof}

Note that the hypothesis of the above lemma 
are satisfied in the case when $X$ is a smooth projective 
variety, $\bar{L}$ a line bundle such that $\bar{L}^\vee$ is 
very ample, and $Y$ is the cone $\hat{L}$ over $X$ with respect 
to the projective imbedding determined by $\bar{L}^\vee$ is 
Cohen-Macaulay at its vertex   
$a \in\hat{L}$.  In this case one says that $X$ is {\it arithmetically Cohen-Macaulay.} For example, if $X$ is the homogeneous 
variety $G/P$ where $G$ is a complex linear algebraic group over $\bc$ and $P$ a parabolic subgroup and $L$ any negative ample line bundle, then the above properties 
hold.  We will also need the fact that $\hat{L}$ is normal---that is, $X=G/P$ is {\it projectively normal} with respect to any ample line bundle.   See \cite{rr} for projective normality and  \cite{ramanathan} for arithmetic Cohen-Macaulayness, where 
these results are established for the more general case of Schubert varieties over arbitrary algebraically closed fields. 
If we assume that $\bar{L}$ itself is very ample, then it is not possible to blow-down 
$X$.  However, in this case, the following lemma allows 
one to compute the cohomology groups of $L$. 

\begin{lemma}
Let $L$ be any holomorphic principal $\bc^*$-bundle over a complex manifold $X$. Then $L\cong L^\vee$ as complex manifolds. In particular, 
$H^{p,q}_{\bar{\partial}}(L)\cong H^{p,q}_{\bar{\partial}}(L^\vee)$.  
\end{lemma}  
\begin{proof}
Let $\psi:L\lr L^\vee$ be the map $v\mapsto v^\vee$ where $v^\vee(\lambda v)=\lambda\in \bc$. Then $\psi$ is a biholomorphism.\end{proof}

Suppose that $L_i\lr X_i, i=1,2,$ are projections of holomorphic principal $\bc^*$-bundles where 
$\bar{L}_i$ are negatively ample over  
complex projective manifolds.  We assume that $X_i$ are 
arithmetically Cohen-Macaulay.  We shall apply the K\"unneth 
formula \cite{cassa} for cohomology with coefficients in coherent analytic sheaves to obtain some 
vanishing results for the cohomology groups $H^q(L;\mathcal{O}_L)$.    
  
A coherent analytic sheaf $\mathcal{F}$ on a complex analytic space $X$ is a  {\it Fr\'echet}  sheaf if 
$\mathcal{F}(U)$ is a Fr\'echet space for any open set $U\subset X$  and if the restriction homomorphism $\mathcal{F}(U)\lr \mathcal{F}(V)$ is continuous for any open set $V\subset U$.  A Fr\'echet sheaf $\mathcal{F}$ is said to be {\it nuclear} if 
$\mathcal{F}(U)$ is a nuclear space for any open set $U$ in $X$.   A Fr\'echet sheaf $\mathcal{F}$ is 
called {\it normal} if there exists a basis for $X$ 
which is a Leray cover for $\mathcal{F}$.  

If $X$ is a complex manifold,  then any coherent analytic sheaf is Fr\'echet-nuclear and normal. See \cite[p. 927]{cassa}.

Let  $V_1,V_2$ are Fr\'echet spaces.  We denote 
by $V_1\otimes_\epsilon V_2$ (resp. $V_1\otimes_\pi V_2$) the inductive (resp. projective) topological tensor 
product of $V_1$ and $V_2$.  If $V_1$ is nuclear, 
then $V_1\otimes_\pi V_2=V_1\otimes _\epsilon V_2$ 
The completed inductive topological tensor product 
tensor product will be denoted $V_1\hat{\otimes}V_2$. 
For a detailed exposition on nuclear spaces see 
\cite{pietsch}. 

One has the notion of completed tensor product $\mathcal{F}\hat{\otimes}\mathcal{G}$ of  coherent analytic Fr\'echet nuclear sheaves $\mathcal{F}, \mathcal{G}$. 
For example, 
the structure sheaf of an analytic space is Fr\'echet nuclear and $\mathcal{O}_{X\times Y}=pr_X^*\mathcal{O}_X \hat{\otimes}pr_Y^*\mathcal{O}_Y$, where $pr_X$ denotes the projection $X\times Y\lr X$. 

We apply the K\"unneth formula for Fr\'echet-nuclear 
normal sheaves established by A. Cassa \cite[Teorema 3]{cassa} to obtain the following

\begin{theorem} \label{cohomologyofl}
(i) Suppose that $L=L_1\times L_2$ where $L_i\lr X_i$ 
are $\bc^*$-bundles over connected complex compact manifolds 
$X_i$ of dimension $\dim X_i\geq 1$. Suppose that, for $i=1,2$,  $\mathcal{F}_i$ is a coherent analytic sheaf over $L_i$ 
such that $H^q(L_i;\mathcal{F}_i), q\geq 0,$ are 
Hausdorff. Then 
\[
H^q(L;\mathcal{F}_1\hat{\otimes}\mathcal{F}_2)\cong \sum_{k+l=q}H^k(L_1;\mathcal{F}_1)\hat{\otimes} H^l(L_2;\mathcal{F}_2)                   \eqno(6)
\] 
for $q\geq 0$.\\  
(ii) Assume that  $X_i$ are projective manifolds and $\bar{L}_i^\vee$ 
are very ample line bundles such that the $\hat{L}_i$  are
Cohen-Macaulay.Then 
$H^q(L;\mathcal{O}_L)=0$  except possibly when $q=0, \dim X_1,\dim X_2, \\ ~\dim X_1+\dim X_2$.  Also, 
$H^0(L;\mathcal{O}_L)\cong H^0(L_1;\mathcal{O}_{L_1})\hat{\otimes}H^0(L_2;\mathcal{O}_{L_2})$ 
\end{theorem}
\begin{proof}  (i) The isomorphism (6)  follows from the K\"unneth formula \cite[Theorema 3]{cassa}.

\noindent (ii)  Let $a_i$ denote the vertex of the cone $\hat{L}_i$ over 
$X_i$ with respect to the projective imbedding 
determined by $L_i^\vee$.  Then $\hat{L}_i$ is an affine 
variety and hence Stein. By \cite[Corollary 2.21, Chapter I]{bs} the cohomology groups 
$H^q(L_i;\mathcal{O}_{L_i})=H^q(\hat{L}_i
\setminus\{a_i\};\mathcal{O}_{\hat{L}_i}), q\geq 0,$ are 
separated and Fr\'echet-Schwartz. Since, by hypothesis 
$\hat{L}_i$ is Cohen-Macaulay at $a_i$, we have $H^q(L_i,\mathcal{O}_{L_i})\cong H^q(\hat{L}_i;\mathcal{O}_{\hat{L}_i})=0$ if $0<q<\dim X_i$ by 
\cite[Theorem 3.6, Chapter II]{bs}.   
The rest of the theorem now follows readily from 
(6).\end{proof}

\begin{remark}{\em  (i)
We remark that the vanishing of the cohomology groups  
$H^q(L;\mathcal{O}_L)$ for $0<q<\min\{\dim X_1,\dim  X_2\}$ in  Theorem \ref{cohomologyofl} (ii) follows from \cite[Ch. I, Theorem 3.6]{bs}. To see this, 
set  $\hat{L}:=\hat{L}_1\times \hat{L}_2\setminus A$ where $A$ is the closed analytic space 
$A=\hat{L}_1\times \{a_2\}\cup \{a_1\}\times \hat{L}_2$. 
The ideal $\mathcal{I}\subset \mathcal{O}_{\hat{L}}$ 
of $A$ equals $\mathcal{I}_1. \mathcal{I}_2$ where 
$\mathcal{I}_1, \mathcal{I}_2$ are the ideals of the components $A_1:=\hat{L}_1\times\{a_2\}, A_2:= \{a_1\}\times \hat{L}_2$ of $A$.  Then  
depth$_A\mathcal{O}_L=\textrm{depth}_{\mathcal{I}}\mathcal{O}_{\hat{L}}=\min_j\{ \textrm{depth}_{\mathcal{I}_j}\mathcal{O}_{\hat{L}}\}=\min_j\{ 
\textrm{depth}_{a_j}\mathcal{O}_{\hat{L}_j}\}
=\min\{\dim X_1+1,\dim X_2+1\}$.
Thus we see that depth$_A\mathcal{O}_{\hat{L}}=\min\{\dim X_1+1,\dim X_2+1\}$. 
Therefore 
$H^q(\hat{L}_1\times \hat{L}_2;\mathcal{O}_{\hat{L}_1\times \hat{L}_2})\cong H^q(L;\mathcal{O}_L)$ if 
$q<\min\{\dim X_1,\dim X_2\}$ by \cite[Ch. I, Theorem 3.6]{bs} where the isomorphism is induced by the inclusion. 
Since $\hat{L}_1\times \hat{L}_2$ is Stein, the 
cohomology groups $H^q(L;\mathcal{O}_L)$ vanish for 
$1\leq q<\min\{\dim X_1,\dim X_2\}$. \\

(ii) When $\pi_1:L_1\lr X_1$ is an {\it algebraic} $\bc^*$-bundle over a smooth projective variety $X_i$, 
one has an isomorphism of quasi-coherent {\it algebraic}  sheaves $\pi_{1,*}(\mathcal{O}^\alg_{L_i})\cong \oplus_{k\in \bz}\bar{L}_i^k$ of $\mathcal{O}_X$-modules.  (By the GAGA principle, $H^*_\alg(X_1;\bar{L}^k_1)\cong H^*(X_1;\bar{L}_1^k)$ for all $k\in \bz$.) 
Therefore the algebraic cohomology groups $H^q_\alg(L_i;\mathcal{O}^\alg)$ can be calculated as 
$H^q_\alg(L_1;\mathcal{O}^\alg_{L_1})\cong H^q(X_1;\pi_{1,*} 
\mathcal{O}_{L_1})\cong \oplus_{k\in \bz} H^q(X_1;\bar{L}^k)$.  If $X_1$ is a flag variety and $\bar{L}_1,$ 
negative ample, then it is known that $H^q(X_1;\bar{L}_1^k)=0$ except when $k> 0$ (resp. $k\leq 0$) and $q=\dim X_1$ (resp. $q=0$).  Furthermore, $H^0(X_1;\bar{L}_1^k)\cong H^{\dim X_1}(X_1;\bar{L}_1^{-k})^\vee$ for $k<0$.  Hence 
$H^q_\alg(L_1;\mathcal{O}^\alg)=0$ unless $q=0,$ or 
$q=\dim X_1$. Now $H^{\dim X_1}_\alg(L_1;\mathcal{O}^\alg)\cong\oplus_{k>0}H^{\dim X_1}(X_1;\bar{L}_1^k)$ and  
$H^0_\alg(L_1;\mathcal{O}^\alg)\cong \oplus_{k\leq 0}(X_1;\bar{L}^{k}_1)$.     We do not know the relation between $H^{\dim X_1}(L_1;\mathcal{O})$ and $H^{\dim X_1}
_\alg(L_1;\mathcal{O}_{L_1}^\alg)$.
}
\end{remark}

Suppose that $\alpha_\lambda$ is an admissible 
$\bc$-action on $L\lr X$ of  scalar type, or diagonal type, or linear type.  It is understood that in the case of diagonal 
type, there is a standard $T_i$-action on $L_i\lr X_i, i=1,2$, and that, in the case of linear type action, $X_i=G_i/P_i$ and $\bar{L}_i$ is negative ample.  Here, and in what follows, the groups $G_i, i=1,2,$ are semi simple and $P_i\subset G_i$ any parabolic subgroups, unless otherwise explicitly stated. 

Denote by $\gamma_\lambda$ (or more briefly $\gamma$) the holomorphic vector field on $L$ associated to the $\bc$-action.  Thus the $\bc$-action is just the flow associated to $\gamma$.  We shall denote by $\mathcal{O}_\gamma^{\textrm{tr}}$ 
the sheaf of germs of local holomorphic functions which are constant along the $\bc$-orbits. Thus 
$\mathcal{O}_\gamma^{\textrm{tr}}$ is isomorphic to $\pi_\lambda^*(\mathcal{O}_{S_\lambda(L)})$.  One has an exact sequence of sheaves 
\[0\to \mathcal{O}_\gamma^{\textrm{tr}}\to \mathcal{O}_L\stackrel{\gamma}{\to}\mathcal{O}_L\to 0.
\eqno(7)\]

Since the fibre of $\pi_\lambda:L\lr S_\lambda(L)$ is Stein, we see that $H^q(L; \mathcal{O}_\gamma^{\textrm{tr}})\cong H^q(S_\lambda(L);\mathcal{O}_{S_\lambda(L)})$ for all $q$.  Thus, 
the exact sequence (7) leads to the following long exact sequence:
\[ 0\to H^0(S_\lambda(L);\mathcal{O}_{S_\lambda(L)})
\to H^0(L;\mathcal{O}_L)\to H^0(L;\mathcal{O}_L)
\lr H^1(S_\lambda(L);\mathcal{O}_{S_\lambda(L)})\to \cdots \]
\[ 
\to H^{q}(S_\lambda(L);\mathcal{O}_{S_\lambda(L)})
\to H^q(L;\mathcal{O}_L)\to H^q(L;\mathcal{O}_L)
\lr H^{q+1}(S_\lambda(L);\mathcal{O}_{S_\lambda(L)})\to \cdots .
\eqno(8)
\]

\begin{theorem}\label{cohomologyofs} With the 
above notations, 
suppose that the $L_i$ satisfy the hypotheses of Theorem \ref{cohomologyofl}(ii) and that $\alpha_\lambda$ is an  admissible $\bc$-action on $L$ of scalar or diagonal or linear type.  Suppose that $1\leq \dim X_1\leq \dim X_2.$  
Then $H^q(S_\lambda(L);\mathcal{O})=0$ provided $q\notin \{0,1, \dim X_i,\dim X_i+1, \dim X_1+\dim X_2, \dim X_1+\dim X_2+1; i=1,2\}.$
Moreover one has $\bc\subset H^1(S_\lambda(L);\mathcal{O})$,  given by 
the constant functions in $H^0(L;\mathcal{O})$.
\end{theorem}
\begin{proof} The only assertion which remains to be explained is 
that the constant function $1$ is not in the image of 
$\gamma_*:H^0(L;\mathcal{O})\lr  H^0(L;\mathcal{O}).$  All other assertions follow trivially from the long exact 
sequence (8) above and Theorem \ref{cohomologyofl}. 

Suppose that $f:L\lr \bc$ is such that $\gamma(f)=1$. This means that $\frac{d}{dz}|_{z=0}(f\circ \mu_p)(z)=1$ 
for all $p\in L, z\in \bc,$ where $\mu_p:\bc\lr L$ is the 
map $z\mapsto \alpha_\lambda(z).p=z.p$. Since $\mu_{w.p}(z)=z.(w.p)=(z+w).p=\mu_p(z+w)$, it follows that $\frac{d}{dz}|_{z=w}(f\circ \mu_p)= 1 ~\forall w\in \bc$. Hence $f\circ\mu_p(z)=z+f(p)$.  This means that the complex hypersurface $Z(f):=f^{-1}(0)\subset L$ meets each fibre at exactly one point.  It follows that the projection $L\lr S_\lambda(L)$ restricts to a  bijection  
$Z(f)\lr S_\lambda(L)$.  

In fact, since 
$\gamma(f)\neq 0$ we see that $Z(f)$ is smooth and since $\gamma_p$ is tangent to the fibres of the projection $L\lr S_\lambda(L)$ for all $p\in Z(f)$,  we see that  the bijective morphism of complex analytic manifolds $Z(f)\lr  S_\lambda(L)$ is an immersion.  It follows that 
$Z(f)\lr S_\lambda(L)$ is a biholomorphism. 
Thus $Z(f)$ is a {\it compact} complex analytic sub manifold of $L\subset \hat{L}.$  Since $\hat{L}$ is Stein, this is a contradiction.   \end{proof}

Our next result concerns the Picard group of $S_\lambda(L)$.

\begin{proposition}\label{picard}
Let $L_i\lr X_i$ be as in Theorem \ref{cohomologyofl} (ii). Suppose that $X_i$ is simply connected.
Then $Pic^0(S_\lambda(L))\cong \bc^l$ for some $l\geq 1$. 
\end{proposition}
\begin{proof}
Since $\bar{L}_i$ is negative ample, $c_1(\bar{L}_i)\in H^2(X_i;\bz)$ is a non-torsion element.   
Clearly $H^1(S_\lambda(L);\mathbb{Z})=0$ by a straightforward argument involving the Serre spectral sequence associated to the principal $\bs^1\times \bs^1$-bundle with projection $S(L)\lr X_1\times X_2$.  Using the exact sequence $0\to \mathbb{Z}\to \mathcal{O}\to \mathcal{O}^*\to 1$  
we see that $Pic^0(S_\lambda(L))\cong H^1(S_\lambda(L);\mathcal{O})\cong \bc^l$.  Now $l\geq 1$ 
by Theorem \ref{cohomologyofs}.  \end{proof}

The above proposition is applicable when $X_i=G_i/P_i$ and $\bar{L}_i$ are negative ample.  However, in this case we have the following stronger result.

\begin{theorem} \label{picard0}  
Let  $X_i=G_i/P_i$ where $G_i$ is semi simple and $P_i$ is any parabolic subgroup and let $\bar{L}_i\lr X_i$ be a  negative line bundle, $i=1,2$.  
We assume that, when $X_i=\bp^1$, the bundle  
$\bar{L}_i$ is a generator of $Pic(X_i)$.   Then 
$Pic^0(S_\lambda(L))\cong \bc$. 
If the $P_i$ are maximal parabolic subgroups and the $L_i$ are generators of $Pic(X_i)\cong \mathbb{Z}$, then $Pic(S_\lambda(L))\cong Pic^0(S_\lambda(L))\cong \bc$. 
\end{theorem}
\begin{proof} It is easy to see that $H^1(S(L);\mathbb{Z})=0$ and that, when $P_i$ are maximal parabolic subgroups and $L_i$ generators of $Pic(X_i)\cong \bz$,   $S(L)$ is $2$-connected. 
If $\dim X_i>1$ for $i=1,2$, then $H^1(L;\mathcal{O})=0$ by Theorem \ref{cohomologyofl} and so we need only show 
that $coker(H^0(L;\mathcal{O})\stackrel{\gamma_*}{\lr} H^0(L;\mathcal{O}))$ is isomorphic to $\bc$. 
In case $\dim X_i=1$---equivalently $X_i=\mathbb{P}^1$---$\bar{L}_i$ is the tautological bundle by our hypothesis. Thus $L_i=\bc^2\setminus\{0\}$.  In this case we need to also show that $\ker(H^1(L;\mathcal{O})\stackrel{\gamma_*}{\lr} H^1(L;\mathcal{O}))$ is zero.  Note that the theorem is known due to Loeb and Nicolau 
\cite[Theorem 2]{ln} when both the $X_i$ are projective spaces and the $\bar{L}_i$ are negative ample generators---in particular when both $X_i=\mathbb{P}^1$.

The validity of the theorem for the case when $\lambda$ is of diagonal type implies its validity in the linear case as well.   This is because one has a 
family $\{L/\bc_{\lambda_\varepsilon}\}$ of complex manifolds parametrized by $\varepsilon\in\bc$ defined  by $\lambda_\varepsilon=\lambda_s+\lambda_{u,\varepsilon}$, where $S_{\lambda_\varepsilon}(L)=L/\bc_{\lambda_\varepsilon}\cong L/\bc_\lambda$ if $\varepsilon\neq 0$ and $\lambda_0:=\lambda_s$ is of diagonal type. (See \S3.)
The semi-continuity property (\cite[Theorem 6, \S4]{ks}) for $\dim H^1(S_{\lambda_\varepsilon}(L);\mathcal{O})$ implies that $\dim H^1(S_\lambda(L);\mathcal{O})\leq 
\dim H^1(S_{\lambda_s}(L);\mathcal{O})$. But Theorem \ref{cohomologyofs} says that $\dim H^1(S_\lambda(L);\mathcal{O})\geq 1$ and so equality must hold, if $H^1(S_{\lambda_s}(L);\mathcal{O})\cong \bc$.  
Therefore we may (and do) assume that the complex structure is of diagonal type. 

First we show that $coker(\gamma_*:H^0(L;\mathcal{O})\lr H^0(L;\mathcal{O}))$ is $1$-dimensional,  generated by the constant functions. Consider the commuting diagram
where $\wt{\gamma}$ is the holomorphic vector field defined by the action of $\bc$ given by $\lambda(\omega_1,\omega_2)$ on $V(\omega_1,\omega_2)$.  Note that $\wt{\gamma}_x=\gamma_x$ if $x\in L$. 

\[
\begin{array}{ccc}
H^0(V(\omega_1,\omega_2);\mathcal{O})& \stackrel{\wt{\gamma}_*}{\lr } &H^0(V(\omega_1,\omega_2);\mathcal{O})\\
\downarrow &                                               & \downarrow\\
 H^0(L;\mathcal{O})& \stackrel{\gamma_*}{\lr} & H^0(L;\mathcal{O})\\
 \end{array}
 \]
By Hartog's theorem,  $H^0(V(\omega_1,\omega_2);\mathcal{O})\cong 
H^0(V(\omega_1)\times V(\omega_2);\mathcal{O})$.  Also, since $\hat{L}_i$ is normal at its vertex \cite{rr}, again by Hartog's theorem,
$H^0(L;\mathcal{O})\cong H^0(\hat{L}_1\times \hat{L}_2;\mathcal{O})$. Since $\hat{L}_i\subset V(\omega_i)$ are closed sub varieties, it follows that 
the both vertical arrows, which are induced by the inclusion of $L$ in $V(\omega_1,\omega_2)$, are surjective. 
From what has been shown in the proof of Theorem \ref{cohomologyofs}, we know that the constant functions 
are not in the cokernel of $\gamma_*$. So it suffices to show that $coker(\wt{\gamma}_*)$ is $1$-dimensional.  This was established in the course of 
proof of Theorem 2 of \cite{ln}.  For the sake of completeness we sketch the proof. 
We identify $V(\omega_i)$ with $\bc^{r_i}$ where $r_i:=
\dim V(\omega_i)$, by choosing a basis for $V(\omega_i)$ consisting of weight vectors. Let $r=r_1+r_2$ so 
that $\bc^r\cong V(\omega_1)\times V(\omega_2)$. 
The problem is reduced to the following: Given a  holomorphic function $f:\bc^r\lr \bc$ with $f(0)=0$, solve for a holomorphic function $\phi$ satisfying the equation 
\[
\sum_{i\leq j} b_{j}z_j\frac{\partial \phi}{\partial z_j}=f,\eqno{(9)}
\] 
where we may (and do) assume that $\phi(0)=0.$
In view of the {\it Observation} made preceding the statement of Theorem \ref{lineartype}, we need only to 
consider the case where 
$(b_{j})\in \bc^r$ satisfies the weak hyperbolicity condition of type $(r_1,r_2)$. 
Denote by $z^{{\bf m}}$ the monomial $z_1^{m_1}\ldots z_n^{m_r}$ where ${\bf m}=(m_1, \ldots,m_r)$ and by $|{\bf m}|$ its degree $\sum_{1\leq j\leq r} m_j.$  Let $f(z)=\sum_{|{\bf m}|>0} a_{\bf m}z^{{\bf m}}\in H^0(\bc^r;\mathcal{O})$.  Then $\phi(z)=\sum a_{{\bf m}}/(b.{\bf m})z^{{\bf m}}$ where $b.{\bf m}=\sum b_{j}m_j$ is the unique solution of Equation (9). 
Note that weak hyperbolicity and the fact that $|{\bf m}|>0$ imply that $b.{\bf m}\neq 0$, and, $b.{\bf m}\to \infty $ as ${\bf m}\to \infty$. Therefore 
$\phi$ is a convergent power series and so $\phi\in H^0(\bc^r;\mathcal{O})$. 

It remains to show that, when $X_1=\mathbb{P}^1$, $L_1=\bc^2\setminus\{0\}$, and 
$\dim X_2>1$,  the homomorphism $\gamma_*:H^1(L;\mathcal{O})\lr H^1(L;\mathcal{O})$ is injective.  
Let $z_j, 1\leq j\leq r,$ denote the coordinates of $\bc^2\times V(\omega_2)$ with respect to a basis consisting of $\wt{T}$-weight vectors.  Since $\dim X_2>1,$ we have $H^1(L_2,\mathcal{O})=0$.  Also $H^1(L_1;\mathcal{O})=H^1(\bc^2\setminus\{0\};\mathcal {O})$ is the space $\mathcal{A}$ of convergent power series $\sum_{m_1, m_2<0} a_{m_1,m_2}z_1^{m_1}z_2^{m_2}$ in $z_1^{-1}, z_2^{-1}$ without constant terms.  

By Theorem \ref{cohomologyofl}, $H^1(L;\mathcal{O})=\mathcal{A}\hat{\otimes} H^0(L_2;\mathcal{O})\cong \mathcal{A}\hat{\otimes} H^0(\hat{L}_2;\mathcal{O})$.  Let $\mathcal{I}\subset H^0(V(\omega_2);\mathcal{O})$ denote the ideal of functions vanishing on $\hat{L}_2$ so that $H^0(\hat{L}_2;\mathcal{O})=H^0(V(\omega_2);\mathcal{O})/\mathcal{I}$.
One has the commuting diagram
\[
\begin{array}{ccccccccc}
0 &\to& \mathcal{A}\hat{\otimes}\mathcal{I} & \to &\mathcal{A}\hat{\otimes}H^0(V(\omega_2);\mathcal{O})
&\to &\mathcal{A}\hat{\otimes}H^0(L_2;\mathcal{O}) &\to  0\\
&&\wt{\gamma}_* \downarrow ~~&& \wt{\gamma}_*\downarrow ~~&&\gamma_*\downarrow ~~&&\\

0 &\to& \mathcal{A}\hat{\otimes}\mathcal{I} & \to &\mathcal{A}\hat{\otimes}H^0(V(\omega_2);\mathcal{O})
&\to &\mathcal{A}\hat{\otimes}H^0(L_2;\mathcal{O}) &\to  0\\
\end{array}
\]
where the rows are exact.  Theorem 2 of \cite{ln} implies that $\wt{\gamma}_*:\mathcal{A}\hat{\otimes}H^0(V(\omega_2);\mathcal {O})\lr  \mathcal{A}\hat{\otimes}H^0(V(\omega_2);\mathcal {O})$ is an 
isomorphism.   
As before, this is equivalent to show that  Equation (9) has a (unique) solution $\phi$ without constant term   when $f=\sum_{{\bf m}}c_{{\bf m}} 
z^{{\bf m}}\in \mathcal{A}\hat{\otimes} H^0(V(\omega_2); \mathcal{O})$, is any convergent power series in $z_1^{-1},z_2^{-1},z_j, 3\leq z_j\leq r,$ where the sum ranges over ${\bf m}=(m_1,m_2,\ldots, m_r)\in \bz^r, m_1,m_2<0, m_j\geq 0,~\forall j\geq 3.$  
It is clear that $\phi(z)
=\sum c_{{\bf m}}/(b.{\bf m})z^{{\bf m}}$ is the unique formal solution.  Note that weak 
hyperbolicity condition implies that $b.{\bf m}\neq 0$ and $b.{\bf m}\to \infty$ as $\sum_{j\geq 1}|m_j|\to \infty$. So $\phi(z)$ is a well-defined convergent power series 
in the variables $z_1^{-1},z_2^{-1}, z_j, j\geq 3$ and is 
divisible by $z_1^{-1}z_2^{-1}$.  Hence $\phi\in \mathcal{A}\hat{\otimes}H^0(V(\omega_2);\mathcal{O})$ and so 
$\wt{\gamma}_*:\mathcal{A}\hat{\otimes}H^0(V(\omega_2);\mathcal{O})\lr\mathcal{A}\hat{\otimes}H^0(V(\omega_2);\mathcal{O})$ is an isomorphism.  
The ideal $\mathcal{I}$ is stable under the action of $\wt{T}_2$, and so is generated as an ideal by (finitely many) polynomials in $z_3,\ldots, z_n$ which are $\wt{T}_2$-weight vectors.  In particular,  the generators are certain homogeneous polynomials $h(z_3,\ldots, z_n)$ such that $\wt{\gamma}_*(z_1^{m_1}z^{m_2}_2h)= b.{\bf m} z_1^{m_1}z_2^{m_2}h~\forall m_1,m_2\in \bz$ where $z^{{\bf m}}$ is any monomial that occurs in $z_1^{m_1}z_2^{m_2}h$.   It follows easily that $\wt{\gamma}_*$ maps  
$\mathcal{A} \hat{\otimes} \mathcal{I}$ isomorphically onto itself.   A straightforward argument involving diagram 
chase now shows that $\gamma_*:\mathcal{A}\hat{\otimes}H^0(L_2;\mathcal{O})\lr\mathcal{A}\hat{\otimes} H^0(L_2;\mathcal{O})$ is an isomorphism.  This completes the proof.  \end{proof}

Assume that $P_i\subset G_i$ are maximal parabolic subgroups so that $Pic(G_i/P_i)\cong \bz$. Suppose that  
the $\bar{L}_i$ are the negative ample generators of the $Pic(G_i/P_i)$.    
We  have the following description of the principal $\bc$-bundles $L_z, z\in \bc$, over $S_\lambda(L)$.
When $z=0$, $L_z$ is the trivial bundle. So let $z\neq 0$.   Let $\{g_{ij}\}$ be a $1$-cocyle defining the principal $\bc$-bundle $L\lr S_\lambda(L)$.  Then the $\bc$-bundle $L_z$ representing the element 
$z[L]\in H^1(S_\lambda(L);\mathcal{O})$ is 
defined by the cocylce $\{zg_{i,j}\}$ for any $z\in \bc$.   We denote the corresponding $\bc$-bundle by $L_z$. Note that the total space and the projection are the same as that 
of $L$. The $\bc$-action on $L_z$ is related to that on $L$  where $w.v\in L_z$ equals $(w/z).v=\alpha_\lambda(w/z)(v)\in L$ for 
$w\in \bc, v\in L$. The vector field corresponding to the $\bc$-action on $L_z$ is given by $(1/z)\gamma_\lambda$.   Of course, when $z=0$, $L_z$ is just the product bundle. 

We shall denote the line bundle (i.e. rank $1$ vector bundle) corresponding to $L_z$ by 
$E_z$. Observe that $E_z=L_z\times _\bc \bc$, where $(w.v, t)\sim (v, \exp(2\pi \sqrt{-1} w)t), ~w, t\in \bc, v\in L_z$, when $z\neq 0$.   If $z\neq 0$, any cross-section $\sigma: S_\lambda(L)\lr E_z=L_z\times_\bc\bc$ corresponds to a holomorphic function $h_\sigma:L\lr \bc$  which satisfies the following:
\[
h_\sigma(w.v)=\exp(-2\pi\sqrt{-1}w)h_\sigma(v) \eqno(10)
\] 
for all $v\in L_z, w\in \bc$.   Equivalently, this means that 
$h_\sigma(\alpha_\lambda(w)v)=\exp(-2\pi\sqrt{-1}wz)h_\sigma(v) $ for $w\in \bc$ and  $v\in L$.
This implies that 
\[
\gamma_\lambda(h_\sigma)=-2\pi\sqrt{-1}z h_\sigma. \eqno(11)
\]
Conversely, if $h$ satisfies (11), then it determines a unique cross-section of $E_z$ over 
$S_\lambda(L)$.  

We have the following result concerning the field of 
meromorphic functions on $S_\lambda(L)$ with 
$\lambda_u=0$.  The proof will be given after some preliminary observations.

\begin{theorem} \label{transcendence} 
Let $L_i$ be the negative ample generator of $Pic(G_i/P_i)\cong \bz$ where $P_i$ is a maximal parabolic subgroup of $G_i$, i=1,2.  Assume that $\lambda_u=0$.
Then the field $\kappa(S_\lambda(L))$ of meromorphic functions of $S_\lambda(L)$ is purely transcendental over $\bc$.  The transcendence degree of $\kappa(S_\lambda(L))$ is less than $\dim S_\lambda(L).$   
\end{theorem}

Let $U_i$ denote the {\it opposite big cell}, namely the  $B_i^{-}$-orbit of $X_i=G_i/P_i$ the identity coset where $B_i^-$ is the Borel 
subgroup of $G_i$ opposed to $B_i$.  
One knows that $U_i$ is 
a Zariski dense open subset of $X_i$ and 
is isomorphic to $\bc^{r_i}$ where $r_i$ is the number of 
positive roots in the unipotent part $P_{i,u}$ of $P_i$. 
The bundle $\pi_i:L_i\lr X_i$ is trivial over $U_i$ and so $\wt{U}_i:=\pi_i^{-1}(U_i)$ is isomorphic to $\bc^{r_i}\times \bc^*$.  We shall now describe a specific isomorphism which will be used in the proof of the above 
theorem. 

Consider the projective imbedding $X_i\subset \mathbb{P}(V(\omega_i))$. 
Let $v_0\in V(\omega_i)$ be a highest weight vector so that 
$P_i$ stabilizes $\bc v_0$; equivalently, $\pi_i(v_0)$ is the identity coset in $X_i$.  Let $Q_i\subset P_i$ be the isotropy at $v_0\in V(\omega_i)$ for the $G_i$ so that  
$G_i/Q_i=L_i$.  The Levi part of $P_i$ is equal to the centralizer of  
a one-dimensional torus $\mathcal{Z}$ contained in $T$ and projects onto $P_i/Q_i\cong \bc^*$, the structure group of $L_i\lr X_i$.  

Let $F_i\in H^0(X_i;L^\vee_i)=V(\omega_i)^\vee$ be the lowest weight vector such that $F_i(v_0)=1$.  Then $U_i\subset X_i$ is precisely the locus $F_i\neq 0$ and $F_i|_{\pi_i^{-1}([v])}:\bc v\lr \bc$ is 
an isomorphism of vector spaces for $v\in\wt{U}_i$.   
We denote also by $F_i$ the restriction of $F_i$ to $\wt{U}_i$. 

Let $Y_\beta$ be the Chevalley basis element of $Lie(G_i)$ of weight $-\beta, \beta\in R^+(G_i)$.  We shall 
denote by $X_\beta\in Lie(G_i)$ the Chevalley basis element of weight $\beta\in R^+(G_i).$  Recall that $H_\beta:=[X_\beta,Y_\beta]\in Lie(T)$ is non-zero whereas 
$[X_\beta,Y_{\beta'}]=0$ if $\beta\neq \beta'$. 

Let $R_{P_i}\subset R^+(G_i)$ denote the set of positive roots of $G_i$ complementary to positive roots of Levi part of $P_i$ and fix an ordering on it.  (Thus $\beta\in R_{P_i}$ if and only if $-\beta$ is a not a root of $P_i$.)  Let $r_i=|R_{P_i}|=\dim X_i$.  Then  $Lie(P_{i,u}^-)\cong \bc^{r_i}$ where $P_{i,u}^-$ denotes the unipotent radical of the 
parabolic subgroup opposed to $P_i$.   Observe that $P_i\cap P_{i,u}^-=\{1\}$.  
The exponential map defines an {\it algebraic} isomorphism $\theta :\bc^{r_i}\cong Lie(P_{i,u})^-\lr U_i$ 
where $\theta((y_\beta)_{\beta\in R_{P_i}})=(\prod_{\beta\in R_{P_i}} \exp(y_\beta Y_\beta)).P_i\in G_i/P_i$.  It is 
understood that, here and in the sequel, the product is carried out according to 
the ordering on $R_{P_i}$.

If $v\in \bc v_0$, then $\theta$ factors through the map 
$\theta_v:\bc^{r_i}\cong Lie(P_{i,u}^-)\lr \wt{U}_i$ defined by $(y_\beta)_{\beta\in R_{P_i}}\mapsto \prod\exp(y_\beta Y_\beta).v$.  Moreover,  $F_i$ is constant ---equal to $F_i(v)$---on the image of $\theta_v$.

We define $\wt{\theta}:\bc^{r_i}\times \bc^*\cong 
Lie(P_{i,u}^-)\times \bc^*\cong P_{i,u}^{-}\times \bc^*\lr \wt{U}_i$ 
to be $\wt{\theta}((y_{\beta}),z)=(\prod\exp(y_\beta Y_\beta)).zv_0=\theta_{z.v_0}((y_\beta))$.  This is an isomorphism.  
We obtain coordinate functions  $z,y_\beta, \beta\in R_{P_i}$ by composing $\wt{\theta}^{-1}$ with  
projections $\bc^{r_i}\times \bc^*\lr \bc$.
Note that $F_i(\wt{\theta}((y_\beta),z)))=z$. Thus the coordinate function $z$ is identified with $F_i$.

Since $F_i$ is the lowest weight vector (of weight $-\omega_i$), $Y_\beta F_i=0$ for all $\beta\in R^+(G_i).$
Define $F_{i,\beta}:=X_\beta(F_i), \beta\in R_{P_i}$.  Then $Y_\beta(F_{i,\beta})=-[X_\beta,Y_\beta]F_i
=-H_\beta(F_i)=\omega_i(H_\beta)F_i$ for all $\beta\in R_{P_i}$. Note that $\omega_i(H_\beta)\neq 0$ as $H_\beta\in R_{P_i}$. 
If $\beta',\beta\in R_{P_i}$ are unequal, then 
$Y_{\beta'}F_{i,\beta}=0$.  It follows that 
$Y_{\beta'}^{m}(F_{i,\beta})=0$ unless $\beta'=\beta$ and $m=1$.

The following result is well-known to experts in 
standard monomial theory.  (See \cite{ls}.)

\begin{lemma} \label{jacobian}
With the above notations, the map $\wt{U}_i\lr \bc^{r_i}\times \bc^*$  defined as 
$v\mapsto ((F_{i,\beta}(v))_{\beta\in R^+_{P_i}}; F_i(v)), ~v\in \wt{U}_i,$ is an algebraic isomorphism for $i=1,2$. 
\end{lemma} 
\begin{proof}  It is easily verified that  
$\partial f/\partial y_\beta|_{v_0}=Y_\beta(f)(v_0)$ for any local holomorphic 
function defined in a neighbourhood of $v_0$.  (Cf. 
\cite{ls}.)

Let $y=\wt{\theta}((y_\gamma),z)=\prod_{\gamma\in R_{P_i}}(\exp(y_\gamma Y_\gamma)\in P^-_i$. 
Denote by $l_y:\wt{U}_i\lr \wt{U}_i$ the left multiplication by $y$. 
If $v=y.v_0\in \wt{U}_i$, then $(\partial /\partial y_\beta|_v)(f)$ equals $(\partial /\partial y_\beta) |_{v_0}
 (f\circ l_y).$
Taking $f=F_{i,\beta},\beta\in R_{P_i}$ 
a straightforward computation using  the observation made preceding the lemma, we see that 
$(\partial /\partial y_\beta|_v)(F_{i,\gamma})=Y_\beta|_{v_0}(F_{i,\gamma}\circ l_y)=F_i(v)\omega_i(H_\beta)\delta_{\beta,\gamma}$ (Kronecker $\delta$).  
We also have 
$(\partial/\partial y_\beta|_v)(F_i)=0$ for all $v\in \wt{U}_i$. Hence $(\partial /\partial y_\beta)|_v(F_{i,\gamma}/F_i)=\omega_i(H_\beta)\delta_{\beta,\gamma}$.  Thus the Jacobian matrix relating the $F_{i,\beta}/F_i$ and 
the $y_\beta,\beta\in R_{P_i},$ is a diagonal 
matrix of {\it constant functions}. The diagonal entries are non-zero as $\omega_i(H_\beta)\neq 0$ for $\beta\in R_{P_i}$ and the lemma follows. 
\end{proof} 


We shall use the coordinate functions $F_i,F_{i,\beta},\beta\in R_{P_i},$ 
to write Taylor expansion for analytic functions on 
$\wt{U}_i$.  In particular, the coordinate ring of 
the affine variety $\wt{U}_i$ is just the 
algebra $\bc[F_{i,\beta},\beta\in R_{P_i}][F_i,F_i^{-1}]$.    The projective normality  \cite{rr}  of $G_i/P_i$ implies that $\bc[\hat{L}_i]=\oplus_{r\geq 0} H^0(X_i;L_i^{-r})=\oplus_{r\geq 0}
V(r\omega_i)^\vee$.  Since $\wt{U}_i$ is defined by the non-vanishing of $F_i$, we see that 
$\bc[\wt{U}_i]=\bc[\hat{L}_i][1/F_i]$.

Now let $X=X_1\times X_2$ and $\wt{T}=\wt{T}_1\times \wt{T}_2\cong (\bc^*)^N, N=n_1+n_2,$ where the isomorphism is as chosen in \S3.  Let $d_i>0,i=1,2,$ be chosen as in Proposition \ref{homogstd} so that the $\wt{T}_i$-action on  
$L_i\lr G_i/P_i$ is $d_i$-standard.   Let $\lambda=\lambda_s\in Lie(\wt{T})$. 
Suppose that $\lambda$ satisfies the weak hyperbolicity condition of type 
$(n_1,n_2)$.

Recall from (4) and (5) that for any weight $\mu_i\in\Lambda(\omega_i)$,  there exist elements  
$\lambda_{\mu_1},\lambda_{\mu_2}\in \bc$ such that for any $v=(v_1,v_2)\in V_{\mu_1}(\omega_1)\times V_{\mu_2}(\omega_2)$, the $\alpha_{\lambda}$-action of $\bc$ is given by $\alpha_{\lambda}(z)v=(\exp(z\lambda_{\mu_1})v_1,\exp(z\lambda_{\mu_2})v_2)$.
In fact $\lambda_{\mu_i}=\sum_{n_{i-1}<j\leq n_{i-1}+n_i} d_{\mu_i,j}\lambda_{j}$ where $d_{\mu_i,j}$ are  certain {\it non-negative integers}.  It follows that, as observed in the discussion preceding the statement of 
Theorem \ref{lineartype}, the complex numbers $\lambda_{\mu_i}\in \bc, ~\mu_i\in \Lambda(\omega_i), i=1,2$ satisfy weak hyperbolicity condition:
\[
0\leq \arg(\lambda_{\mu_1})<\arg(\lambda_{\mu_2})<\pi, ~\forall \mu_i\in \Lambda(\omega_i), i=1,2.\eqno(12)
\]

We observe that if $\mu=\mu_1+\cdots+\mu_r=\nu_1+\cdots+\nu_r$, where $\mu_j,\nu_j\in \Lambda(\omega_i)$, then $\lambda_{\mu,r}:=\sum \lambda_{\mu_j}=\sum \lambda_{\nu_j}$. (This is a straightforward verification 
using (3) and (4).)
Therefore, if $v\in V(\omega_i)^{\otimes r}$ is any 
weight vector of weight $\mu$, we get, for the 
diagonal action of $\bc$, 
$z.v=\exp(\lambda_{\mu,r}z)v$.

Any finite dimensional $\wt{G}_i$-representation space $V$ is naturally $\wt{G}_1\times \wt{G}_2$-representation space and is a direct sum of its $\wt{T}$-weight spaces $V_\mu$. If $V$ arises from a representation of $G_i$ via $\wt{G}_i\lr G_i$, then 
the $\wt{T}$-weights of $V$ are the same as $T$-weights.

\begin{definition}
Let $Z_i(\lambda)\subset \bc, i=1,2,$ be the abelian 
subgroup generated by $\lambda_{\mu}, \mu\in \Lambda(\omega_i)$ and  let $Z(\lambda):=Z_1(\lambda)+ Z_2(\lambda)\subset \bc$.  

The $\lambda$-{\em weight} of an element $0\neq f\in 
Hom(V_\mu(\omega_i); \bc)$ is defined to be $wt_\lambda(f):=\lambda_{\mu}$.  If $h\in 
Hom(V(\omega_i)^{\otimes r},\bc )$ is a weight vector of weight $-\mu$, (so that $h\in Hom(V(\omega_i)^{\otimes r}_\mu;\bc)$) we define the $\lambda$-{\em weight of } of $h$ to be $\lambda_{\mu,r}$.  

If $f\in Hom(V_\mu(r\omega_i),\bc)$ is a weight vector (of weight $-\mu$), then it is the image of a unique weight vector $\wt{f}\in Hom(V(\omega_i)^{\otimes r},\bc)$ under the surjection induced by the $\wt{G}_i$-inclusion 
$V(r\omega_i)\hookrightarrow V(\omega_i)^{\otimes r}=V(r\omega_i)\oplus V'$ where $\wt{f}|V'=0$.  We define the $\lambda$-{\em weight of} $f$ to be $wt_\lambda(f):=wt_\lambda(\wt{f})$.  
\end{definition}

If $h_i\in V(r_i\omega_i)^\vee\subset \bc[\hat{L}_i]$, $i=1,2,$ are weight vectors, then $h_1h_2$ is a weight vector of $V(r_1\omega_1)^\vee\otimes V(r_2\omega_2)^\vee\subset \bc[\hat{L}_1\times \hat{L}_2]$ and we have 
$wt_\lambda(h_1h_2)=wt_\lambda(h_1)+wt_\lambda(h_2)\in Z(\lambda).$ Note that $wt_\lambda(f_1\ldots f_k)=\sum_{1\leq j\leq k}wt_\lambda(f_j)\in Z(\lambda)$ where $f_j\in \bc[\hat{L}_1\times \hat{L}_2]=\oplus_{r_1,r_2\geq 0} V(r_1\omega_1)^\vee\otimes V(r_2\omega_2)^\vee$ are weight vectors.  

Also $wt_\lambda(f)\in Z(\lambda)$ is a {\em non-negative} linear combination of $\lambda_j, 1\leq j\leq N,$
for any $\wt{T}$-weight vector $f\in \bc[\hat{L}_1\times \hat{L}_2]$.

If $f\in V(\omega_i)^\vee$, it defines a holomorphic function on $V(\omega_1)\times V(\omega_2)$ and hence on $L$, and denoted by the same symbol $f$; explicitly $f(u_1,u_2)=f(u_i),~\forall (u_1,u_2)\in L.$

\begin{lemma} \label{eigenvectors}
We keep the above notations.  Assume that $\lambda=\lambda_s\in Lie(\wt{T})=\bc^N$. 
     
Fix $\bc$-bases $\mathcal{B}_i$ for 
$V(\omega_i)^\vee$,   
consisting of $\wt{T}$-weight vectors. Let $z_0\in Z(\lambda)$. 
There are only finitely many monomials $f:=f_1\ldots f_k$, $f_j\in \mathcal{B}_1\cup \mathcal{B}_2$  
having $\lambda$-weight $z_0$.  Furthermore, $v_\lambda(f)=wt_\lambda(f)f$.   
\end{lemma}

\begin{proof}
The first statement is a consequence of weak hyperbolicity (see (12)).  Indeed, since $0\leq \arg(\lambda_{\mu})<\pi$  for all $\mu\in \Lambda(\omega_i), i=1,2$, given any 
complex number $z_0$, there are only finitely many non-negative integers $c_j$ such that 
$\sum c_j\lambda_{\mu_j}=z_0$.  

As for the second statement, we need only verify this for $f\in \mathcal{B}_i, i=1,2.$ Suppose that 
$f\in \mathcal{B}_1$ and that $f$ is of  weight $-\mu$, $\mu\in \Lambda(\omega_1)$, say. Then, for any 
$(u_1,u_2)\in L$, writing 
$u_1=\sum_{\nu\in\Lambda(\omega_1)}u_\nu$, using linearity and the fact that $f(u_1,u'_2)=f(u_\mu)$ we have
\[
\begin{array}{lll}
v_\lambda(f)(u_1,u_2)&=&\lim_{w\to 0}(f(\alpha_\lambda(w)(u_1,u_2))-f(u_1,u_2))/w\\
&=&\lim_{w\to 0} (f(\exp(\lambda_\mu w)u_\mu)-f(u_1))/w\\
&=& \lim_{w\to 0}(\frac{\exp(\lambda_\mu w)-1}{w})f(u_1)\\
&=&\lambda_\mu f(u_1)\\
&=&\lambda_\mu f(u_1,u_2).
\end{array}
\]
This completes the proof.
\end{proof}

We assume that $F_i,F_{i,\beta},\beta\in R_{P_i}$, are in $\mathcal{B}_i$, $i=1,2$. 

Let $\mathcal{M}\subset \bc(\wt{U}_1\times \wt{U}_2)$ denote the multiplicative group of all Laurent monomials 
in $F_i,F_{i,\beta},\beta\in R_{P_i}, i=1,2$.  One has a homomorphism 
$wt_\lambda:\mathcal{M}\lr Z(\lambda)$. Denote by $\mathcal{K}$ the kernel of $wt_\lambda$.   Evidently, $\mathcal{M}$ is a free abelian group of rank $\dim L$.

\begin{lemma} \label{independence} 
With the above notations, $wt_\lambda:\mathcal{M}\lr Z(\lambda)$ is surjective.  Any $\bz$-basis $h_1,\ldots, h_k$ of $\mathcal{K}$ is algebraically independent over $\bc$. 
\end{lemma}

\begin{proof}  Suppose that $\nu\in Z_i(\lambda)$. Write 
$\nu=\sum a_\mu\lambda_\mu$ and choose $b_\mu\in \mathcal{B}_i$ to be of weight $\mu$.  
Then $wt_\lambda(\prod_{\mu} b_\mu^{a_\mu})=\nu$.  
On the other hand, $wt_\lambda(b_\mu)$ equals the $\lambda$-weight of any monomial in the 
$F_i^{-1},F_i,F_{i,\beta},\beta\in R_{P_i}$ that occurs in $b_\mu|\wt{U}_i$.  The first assertion follows from this. 

Let, if possible, $P(z_1,\ldots, z_k)=0$ be a polynomial equation satisfied by $h_1,\ldots, h_k$.  Note that the $h_j$ are certain Laurent monomials in a transcendence basis 
of the field $\bc(\wt{U}_1\times\wt{U}_2)$ of {\it rational functions} on the affine variety $\wt{U}_1\times \wt{U}_2$. Therefore there must exist monomials  $z^{{\bf m}}$ and $z^{{\bf m}'}$, ${\bf m}\neq {\bf m}'$,   occurring  in $P(z_1,\ldots, z_k)$ with non-zero coefficients such that $h^{{\bf m}}=h^{{\bf m}'}\in \bc(\wt{U}_1\times \wt{U}_2)$.  Hence $h^{{\bf m}-{\bf m}'}=1$. This contradicts the hypothesis that the $h_j$ are linearly independent in the multiplicative group $\mathcal{K}$.    
\end{proof}

We now turn to the proof of Theorem \ref{transcendence}.

\noindent
{\it Proof of Theorem \ref{transcendence}:}
By definition, any meromorphic function on $S_\lambda(L)$ is a quotient  $f/g$ where $f$ and $g$ are 
holomorphic sections of a holomorphic line bundle $E_z$.  Any holomorphic section $f:S(L)\lr E_z$ 
defines a holomorphic function on $L$, denoted by $f$, which satisfies Equation (11).
By the normality of $\hat{L}_1\times \hat{L}_2$, the function $f$ then extends uniquely to a function on $\hat{L}_1\times \hat{L}_2$ which is 
again denoted $f$.   Thus we may write $f=\sum_{r,s\geq 0} f_{r,s}$ where $f_{r,s}\in V(r\omega_1)^\vee\otimes V(s\omega_2)^\vee$.  Now $ v_{\lambda}f=a f$ and $v_\lambda f_{r,s}\in V(r\omega_1)^\vee\otimes V(s\omega_2)^\vee$ implies that $v_\lambda(f_{r,s})=a f_{r,s}$ for all $r,s\geq 0$ where $a=-2\pi\sqrt{-1}z$.  This implies that $wt_\lambda(f_{r,s})=a$ for all $r,s\geq 0$.   This implies, by Lemma \ref{eigenvectors}, 
that $f_{r,s}=0$ for sufficiently large $r,s$ and so $f$ is {\it algebraic}.

Now writing  $f$ and $g$ restricted to  $\wt{U}_1\times \wt{U}_2$  as a polynomial in the  
the coordinate functions  $F_i^{\pm}, F_{i,\beta}, i=1,2,$ introduced above, 
it follows easily that $f/g$ belongs to the field  $\bc(\mathcal{K})$ generated by $\mathcal{K}$. 
Evidently $\mathcal{K}$---and hence the field $\bc(\mathcal{K})$---is contained in $\kappa(S_\lambda(L))$. 
Therefore $\kappa(S_\lambda(L))$ equals $\bc(\mathcal{K})$. By Lemma \ref{independence} the field $\bc(\mathcal{K})$ is purely transcendental over $\bc$.  

Finally, since $Z(\lambda)$ is of rank at least $2$ and since $wt_\lambda:\mathcal{M}\lr Z(\lambda)$ is surjective, $tr. deg(\kappa(S_\lambda(L))=rank(\mathcal{K})\leq rank(\mathcal{K})
-2=\dim(L)-2=\dim(S_\lambda(L))-1$.  \hfill $\Box$

\begin{remark}{\em 
(i)   We have actually shown that the transcendence degree of $\kappa(S_\lambda(L))$ equals the 
rank of $\mathcal{K}$.  In the case when $X_i$ are projective spaces, this was observed by \cite{ln}.   When $\lambda$ is of scalar type, $tr. deg(\kappa(S_\lambda(L))=\dim (S_\lambda(L))-1$. \\
(ii) Theorem \ref{transcendence} implies that any algebraic reduction of $S_\lambda(L)$ is a rational 
variety.  In the case of scalar type, one has an elliptic 
curve bundle $S_\lambda(L)\lr X_1\times X_2$. 
(Cf. \cite{sankaran}.)  Therefore this bundle projection 
yields an algebraic reduction. 
In the general case however, it is an interesting problem 
to construct explicit algebraic reductions of these 
compact complex manifolds.  (We refer the reader to \cite{peternell} and references therein to 
basic facts about algebraic reductions.)\\
(iii)  We conjecture that $\kappa(S_\lambda(L))$ is 
purely transcendental for $X_i=G_i/P_i$ where $P_i$ is 
any parabolic subgroup and $\bar{L}_i$ is any negative ample line bundle over $X_i$, where $S_\lambda(L)$ has 
any linear type complex structure.  
} 
\end{remark}

\noindent
{\bf Acknowledgments:}  The authors thank D. S. Nagaraj for helpful discussions and for his valuable comments.  Also the authors thank the referee for his/her critical comments which resulted in 
improved clarity of exposition. 


\end{document}